\documentclass[11pt]{article}
\usepackage[letterpaper,hmargin=1in,vmargin=1.25in]{geometry}

\usepackage{url}

\usepackage{amsmath}
\usepackage{amssymb}

\usepackage{amsthm}

\theoremstyle{plain}
\newtheorem{theorem}{Theorem}[section]
\newtheorem{lemma}[theorem]{Lemma}

\newtheorem{proposition}[theorem]{Proposition}
\newtheorem{definition}[theorem]{Definition}

\newtheorem{replacements}{Replacements}

\newcommand{\abs}[1]{\left\lvert#1\right\rvert}

\DeclareMathOperator{\Dom}{dom}
\newcommand{\rest}[2]{#1\!\!\restriction_{#2}}
\newcommand{\reste}[2]{#1\restriction_{#2}}

\newcommand{\N}{\mathbb{N}}
\newcommand{\Z}{\mathbb{Z}}
\newcommand{\Q}{\mathbb{Q}}
\newcommand{\R}{\mathbb{R}}
\newcommand{\X}{\{0,1\}^*}

\DeclareMathOperator{\Pf}{Pf}

\newcommand{\noi}{\noindent}

\title{\textbf{
A Computational Complexity-Theoretic Elaboration of Weak Truth-Table Reducibility%
\thanks{
A preliminary version of this work was presented
under the title ``One-wayness and two-wayness in algorithmic randomness'',
at the 5th Conference on Logic, Computability and Randomness,
May 24-28, 2010, University of Notre Dame,
Notre Dame,
Indiana, USA.
}
}}

\author{
Kohtaro Tadaki\\
\\
Research and Development Initiative, Chuo University\\
JST CREST\\
1--13--27 Kasuga, Bunkyo-ku, Tokyo 112-8551, Japan\\
E-mail: tadaki@kc.chuo-u.ac.jp\\
http://www2.odn.ne.jp/tadaki/
}

\date{}

\begin{document}

\maketitle

\begin{quotation}
\noi\textbf{Abstract.}
The notion of weak truth-table reducibility plays an important role in recursion theory.
In this paper, we introduce an elaboration of this notion, where a computable bound on the use function is explicitly specified.
This elaboration enables us to deal with the notion of asymptotic behavior in a manner like in computational complexity theory,
while staying in computability theory.
We apply the elaboration to sets which appear in the statistical mechanical interpretation of algorithmic information theory.
We demonstrate the power of the elaboration by revealing a critical phenomenon, i.e., a phase transition, in the statistical mechanical interpretation,
which cannot be captured by the original notion of weak truth-table
reducibility.
\end{quotation}

\vspace{1mm}

\begin{quotation}
\noi\textit{Key words\/}:
algorithmic information theory,
algorithmic randomness,
weak truth-table reducibility,
Chaitin $\Omega$ number,
partition function,
halting problem,
phase transition,
statistical mechanics,
computational complexity theory,
program-size complexity
\end{quotation}

\section{Introduction}

The notion of weak truth-table reducibility plays an important role in recursion theory
(see e.g.~\cite{O89,N09,DH10}).
For any sets $A,B\subset\N$, we say that \textit{$A$ is weak truth-table reducible to $B$}, denoted $A\le_{wtt} B$,
if there exist an oracle Turing machine $M$ and a total recursive function $g\colon\N\to\N$ such that
$A$ is Turing reducible to $B$ via $M$ and, on every input $n\in\N$, $M$ only queries natural numbers at most $g(n)$.
In this paper, we introduce
an elaboration of
this notion,
where the total recursive bound $g$ on the use
of the reduction
is explicitly specified.
In doing so, in particular
we try to follow the fashion in which
computational complexity theory is developed,
while staying in computability theory.
We apply the elaboration to sets which appear in the theory of program-size, i.e., algorithmic information theory (AIT, for short) \cite{C87b,C02,N09,DH10}.
The elaboration,
called \textit{
reducibility
in query size $f$},
is introduced as follows.

\begin{definition}[
reducibility
in query size $f$]\label{new reduction}
Let $f\colon\N\to\N$, and let $A,B\subset\X$.
We say that
\textit{$A$ is reducible to $B$ in query size $f$}
if there exists an oracle deterministic Turing machine $M$ such that
\begin{enumerate}
  \item $A$ is Turing reducible to $B$ via $M$, and
  \item on every input $x\in\X$,
    $M$ only queries strings of length at most $f(\abs{x})$.\qed
\end{enumerate}
\end{definition}

For any fixed sets $A$ and $B$,
the above definition allows us
to consider
the notion of asymptotic behavior for the function $f$
which bounds
the use of the reduction,
i.e.,
which
imposes
the restriction on the use of
the computational resource
(i.e., the oracle $B$).
Thus,
by the above definition,
even in the context of computability theory,
we can deal with the notion of asymptotic behavior
in a manner like in computational complexity theory.
Recall here that
the notion of input size
plays a crucial role in computational complexity theory
since computational complexity such as time complexity and space complexity
is measured based on it.
This is also true in AIT
since the program-size complexity is measured based on input size.
Thus, in Definition~\ref{new reduction} we
consider a reduction between subsets of $\X$
and not a reduction between subsets of $\N$
as in the original weak truth-table
reducibility.
Moreover,
in Definition~\ref{new reduction}
we require the bound $f(\abs{x})$ to depend only on input size $\abs{x}$
as in computational complexity theory,
and not on input $x$ itself
as in the original weak truth-table reducibility.
We pursue
a formal
correspondence to computational complexity theory
in this manner,
while staying in computability theory.

In this paper we demonstrate the power of the notion of
reducibility
in query size $f$ in the context of AIT.
In \cite{C75} Chaitin introduced $\Omega$ number as a concrete example of random real.
His $\Omega$ is defined as the probability that an optimal prefix-free machine $U$ halts, and plays a central role in the development of AIT.
Here the notion of \textit{optimal prefix-free machine} is used to define the notion of \textit{program-size complexity} $H(s)$ for a finite binary string $s$.
The first $n$ bits of the base-two expansion of $\Omega$ solve the halting problem of the optimal prefix-free machine $U$ for all binary inputs of length at most $n$.
Using this property, Chaitin showed $\Omega$ to be a random real.
Let $\Dom U$ be the set of all halting inputs for $U$.
Calude and Nies \cite{CN97}, in essence, showed the following theorem on the relation between the base-two expansion of $\Omega$ and the halting problem $\Dom U$.

\begin{theorem}[Calude and Nies \cite{CN97}]\label{CN}
$\Omega$ and $\Dom U$ are weak truth-table equivalent.
Namely, $\Omega\le_{wtt}\Dom U$ and $\Dom U\le_{wtt}\Omega$.
\qed
\end{theorem}

In \cite{T02}
we generalized $\Omega$ to $Z(T)$
by
\begin{equation}\label{partition_function}
  Z(T)=\sum_{p\in\Dom U}2^{-\frac{\abs{p}}{T}}
\end{equation}
so that
the partial randomness of $Z(T)$ equals to $T$ if $T$ is a computable real with $0<T\le 1$.%
\footnote{%
In \cite{T02},
$Z(T)$ is denoted by $\Omega^T$.}
Here the notion of \textit{partial randomness} of a real is a stronger representation of the compression rate of the real by means of program-size complexity.
The real function $Z(T)$ of $T$ is a function of class $C^\infty$ on $(0,1)$ and an increasing continuous function on $(0,1]$.
In the case of $T=1$,
$Z(T)$
results in
$\Omega$, i.e., $Z(1)=\Omega$.
We can show Theorem~\ref{CNT} below for $Z(T)$.
This theorem follows immediately from stronger results, Theorems~\ref{main1} and \ref{main2}, which are two of the main results of this paper.

\begin{theorem}\label{CNT}
Suppose that $T$ is a computable real with
$0<T<1$.
Then
$Z(T)$ and $\Dom U$ are weak truth-table equivalent.
\qed
\end{theorem}

When comparing Theorem~\ref{CN} and Theorem~\ref{CNT},
we see that
there is no difference between $T=1$ and $T<1$
with respect to
the weak truth-table equivalence between $Z(T)$ and $\Dom U$.
In this paper, however, we show that there is a critical difference between $T=1$ and $T<1$ in the relation between $Z(T)$ and $\Dom U$
from the point of view of the reducibility in query size $f$.
Based on the notion of reducibility in query size $f$,
we introduce the notions of \textit{unidirectionality} and \textit{bidirectionality} between two sets $A$ and $B$ in this paper.
These notions enable us to investigate the relative computational power between $A$ and $B$.

Theorems~\ref{one-way I} and \ref{one-way II} below are
two of the main results of this paper.
Theorem~\ref{one-way I} gives a succinct equivalent characterization of $f$
for which $\Omega$ is reducible to $\Dom U$ in query size $f$
and
reversely
Theorem~\ref{one-way II} gives a succinct equivalent characterization of $f$
for which $\Dom U$ is reducible to $\Omega$ in query size $f$,
both
in a general setting.
Based on them,
we
show in Theorem~\ref{one-wayness_in_both_directions} below
that the computation from $\Omega$ to $\Dom U$ is unidirectional
and the computation from $\Dom U$ to $\Omega$ is also unidirectional.
On the other hand,
Theorems~\ref{main1} and \ref{main2} below are also
two of the main results of this paper.
Theorem~\ref{main1} gives a succinct equivalent characterization of $f$
for which $Z(T)$ is reducible to $\Dom U$ in query size $f$
and
reversely
Theorem~\ref{main2} gives a succinct equivalent characterization of $f$
for which $\Dom U$ is reducible to $Z(T)$ in query size $f$,
both
in a general setting,
in the case where $T$ is a computable real with $0<T<1$.
Based on them,
we show in Theorem~\ref{two-wayness_for_T<1} below that
the computations between $Z(T)$ and $\Dom U$ are bidirectional
if $T$ is a computable real with $0<T<1$.
In this
way
the notion of
reducibility
in query size $f$
can reveal
a critical difference of the behavior of $Z(T)$ between $T=1$ and $T<1$,
which cannot be captured by the original notion of weak truth-table
reducibility.

In our former work \cite{T09CiE} we considered
some
elaboration of weak truth-table equivalence between $\Omega$ and $\Dom U$ and showed the unidirectionality between them in a certain form.
Compared with
this
paper, however, the treatments of \cite{T09CiE} were insufficient in the correspondence to computational complexity theory.
In this paper, based on the notion of
reducibility
in query size $f$, we sharpen the results of \cite{T09CiE}
with a thorough emphasis on
a formal correspondence to computational complexity theory.

\subsection{Statistical Mechanical Interpretation of AIT as Motivation}
\label{smiait}

In this subsection we explain the motivation of this work.
The readers can skip this subsection if they are not interested in the motivation.

In \cite{T08CiE}
we introduced and developed
the statistical mechanical interpretation of AIT.
We there introduced
\textit{the thermodynamic quantities at temperature $T$},
such as partition function $Z(T)$, free energy $F(T)$, energy $E(T)$, statistical mechanical entropy $S(T)$, and specific heat $C(T)$,
into AIT.
These quantities are real functions of a real argument $T>0$,
and are introduced based on
$\Dom U$
in the following manner.

In statistical mechanics,
the partition function $Z_{\mathrm{sm}}(T)$,
free energy $F_{\mathrm{sm}}(T)$,
energy $E_{\mathrm{sm}}(T)$,
entropy $S_{\mathrm{sm}}(T)$,
and specific heat $C_{\mathrm{sm}}(T)$
at temperature $T$
are given
as follows:
\begin{equation}\label{tdqsm}
\begin{split}
  Z_{\mathrm{sm}}(T)
  &=\sum_{x\in X}e^{-\frac{E_x}{k_{\mathrm{B}}T}},
  \hspace*{27.5mm}
  F_{\mathrm{sm}}(T)
  =-k_{\mathrm{B}}T\ln Z_{\mathrm{sm}}(T), \\
  E_{\mathrm{sm}}(T)
  &=\frac{1}{Z_{\mathrm{sm}}(T)}\sum_{x\in X}E_xe^{-\frac{E_x}{k_{\mathrm{B}}T}},
  \hspace*{9mm}
  S_{\mathrm{sm}}(T)
  =\frac{E_{\mathrm{sm}}(T)-F_{\mathrm{sm}}(T)}{T}, \\
  &\hspace*{25mm}
  C_{\mathrm{sm}}(T)
  =\frac{d}{dT}E_{\mathrm{sm}}(T),
\end{split}
\end{equation}
where $X$ is a complete set of energy eigenstates of
a quantum system
and $E_x$ is the energy of an energy eigenstate $x$.
The constant $k_{\mathrm{B}}$ is called \textit{the Boltzmann Constant},
and the $\ln$ denotes the natural logarithm.
For
the meaning and importance of these thermodynamic quantities in statistical mechanics,
see e.g.~Chapter 16 of \cite{C85} or Chapter 2 of \cite{TKS92}.%
\footnote{
To be precise,
the partition function is not a thermodynamic quantity
but a statistical mechanical quantity.}

In \cite{T08CiE} we introduced
thermodynamic quantities
into AIT
by performing Replacements~\ref{CS06} below
for the thermodynamic quantities \eqref{tdqsm} in statistical mechanics.

\begin{replacements}\label{CS06}\hfill
\begin{enumerate}
  \item Replace the complete set $X$ of energy eigenstates $x$
    by the set $\Dom U$ of all programs $p$ for $U$.
  \item Replace the energy $E_x$ of an energy eigenstate $x$
    by the length $\abs{p}$ of a program $p$.
  \item Set the Boltzmann Constant $k_{\mathrm{B}}$ to $1/\ln 2$.%
    \footnote{The so-called Boltzmann's entropy formula has the form $S_{\mathrm{sm}}=k_{\mathrm{B}}\ln W$,
    where $W$ is the number of microstates consistent with a given macrostate.
    By setting $k_{\mathrm{B}}=1/\ln 2$,
    the Boltzmann formula results in the form $S_{\mathrm{sm}}=\log_2 W$.
    Thus, since the logarithm is to the base $2$ in the resultant formula, Replacements~\ref{CS06} (iii) is
    considered to be natural from the points of view of AIT and classical information theory.}\qed
\end{enumerate}
\end{replacements}

For example, based on Replacements~\ref{CS06}, the partition function $Z(T)$ at temperature $T$ is introduced from \eqref{tdqsm} as
$Z(T)=\sum_{p\in\Dom U}2^{-\abs{p}/T}$.
This is precisely $Z(T)$ defined by \eqref{partition_function}.
In general, the thermodynamic quantities in AIT are variants of Chaitin $\Omega$ number.

In \cite{T08CiE} we proved that
if the temperature $T$ is a computable real with $0<T<1$
then, for each of the thermodynamic quantities $Z(T)$, $F(T)$, $E(T)$, $S(T)$, and $C(T)$, the partial randomness of its value equals to $T$.
Thus,
the temperature $T$ plays a role as
the partial randomness (and therefore the compression rate) of all the thermodynamic quantities
in the statistical mechanical interpretation of AIT.
In \cite{T08CiE}
we further showed that
the temperature $T$ plays a role as the partial randomness of
the temperature $T$ itself,
which is a thermodynamic quantity of itself
in thermodynamics or statistical mechanics.
Namely,
we proved
\textit{the fixed point theorem for partial randomness},%
\footnote{The fixed point theorem for partial randomness is called a fixed point theorem on compression rate in \cite{T08CiE}.}
which states that, for every $T\in(0,1)$,
if the value of the partition function $Z(T)$ at temperature $T$ is a computable real,
then the partial randomness of $T$ equals to $T$, and therefore the compression rate of $T$ equals to $T$, i.e.,
$\lim_{n\to\infty}H(\rest{T}{n})/n=T$, where $\rest{T}{n}$ is the first $n$ bits of the base-two expansion of $T$.

In our second work \cite{T09LFCS} on the
interpretation,
we showed that
a fixed point theorem
of the same form as for $Z(T)$ holds also for each of free energy $F(T)$, energy $E(T)$, and statistical mechanical entropy $S(T)$.
Moreover,
based on
the statistical mechanical relation $F(T)=-T\log_2 Z(T)$,
we showed that the computability of $F(T)$ gives completely different fixed points from the computability of $Z(T)$.

In the third work \cite{T09ITW},
we pursued the formal correspondence between the statistical mechanical interpretation of AIT and normal statistical mechanics further,
and then unlocked the properties of the sufficient conditions (i.e., the computability of $Z(T)$, $F(T)$, $E(T)$, or $S(T)$ for $T$)
for the fixed points for partial randomness further.
Recall that the thermodynamic quantities in AIT are defined based on the domain of definition of an optimal prefix-free machine $U$.
In \cite{T09ITW}, we showed that
there are infinitely many optimal prefix-free machines which
give completely different sufficient conditions in all of the thermodynamic quantities in AIT.
We did this by introducing the notion of composition of prefix-free machines
into
AIT, which corresponds to the notion of composition of systems in normal statistical mechanics.

How are Replacements~\ref{CS06} justified ?
Generally speaking,
in order to give a statistical mechanical interpretation to a framework
which looks unrelated to statistical mechanics at first glance,
it is important to identify a \textit{microcanonical ensemble}
in the framework.
Once we can do so,
we can easily develop an equilibrium statistical mechanics on the framework
according to the theoretical development of
normal equilibrium statistical mechanics.
Here, the microcanonical ensemble is a certain sort of
uniform probability distribution.
In fact,
in the work \cite{T07}
we developed
a statistical mechanical interpretation of
the noiseless source coding scheme
in information theory
by identifying a microcanonical ensemble in the scheme.
Then,
based on this identification,
in \cite{T07}
the notions in statistical mechanics
such as statistical mechanical entropy, temperature,
and thermal equilibrium are translated into the context of
noiseless source coding.

Thus,
in order to develop a total statistical mechanical interpretation of AIT,
it is appropriate to identify a microcanonical ensemble in the framework of
AIT.
Note, however, that
AIT is not a physical theory but a purely mathematical theory.
Therefore,
in order to obtain significant results
for the development of AIT itself,
we have to develop
a statistical mechanical interpretation
of AIT
in a mathematically rigorous manner,
unlike in normal statistical mechanics in physics
where arguments are not necessarily mathematically rigorous.
A fully rigorous mathematical treatment of statistical mechanics is
already developed (see Ruelle \cite{Ru99}).
At present,
however,
it would not as yet seem to be an easy task to merge AIT
with this mathematical treatment in a satisfactory manner.
In our former works \cite{T08CiE,T09LFCS,T09ITW} mentioned above,
for mathematical strictness
we developed a statistical mechanical interpretation of AIT
in a different way from
the idealism above.
We there introduced the thermodynamic quantities at temperature $T$
into AIT
by performing Replacements~\ref{CS06}
for the corresponding thermodynamic quantities \eqref{tdqsm} at temperature $T$
in statistical mechanics.
We then obtained the various rigorous results, as reviewed in the above.

On the other hand, in the work \cite{T10JPCS} we showed that,
if we do not stick to the mathematical strictness of an argument,
we can certainly develop a total statistical mechanical interpretation of AIT
which
attains
a perfect correspondence to normal statistical mechanics.
In the total interpretation,
we identify a microcanonical ensemble in AIT in a similar manner to \cite{T07},
based on the probability measure which gives
Chaitin $\Omega$ number
the meaning of the halting probability actually.
This identification
clarifies
the
meaning of the thermodynamic quantities of AIT,
which are originally introduced by \cite{T08CiE} in a rigorous manner
based on Replacements~\ref{CS06}.

In the present paper, we continue the rigorous treatment of the statistical mechanical interpretation of AIT performed
by our former works \cite{T08CiE,T09LFCS,T09ITW}.
As a result, we reveal a new aspect of the thermodynamic quantities of AIT.
The work \cite{T08CiE} showed that
the values of all the thermodynamic quantities,
including $Z(T)$,
diverge when the temperature $T$ exceeds $1$.
This phenomenon
may
be regarded as
\textit{phase transition} in statistical mechanics.
The present
paper reveals
a new aspect of the phase transition
by showing
the critical difference of the behavior of $Z(T)$ between $T=1$ and $T<1$
in terms of
reducibility
in query size $f$.

\subsection{Organization of the Paper}

We begin in Section~\ref{preliminaries} with some preliminaries to AIT and partial randomness.
In Section~\ref{elaboration} we investigate
simple properties of the notion of reducibility in query size $f$ and introduce the notions of unidirectionality and bidirectionality between two sets based on it.
We then show in Section~\ref{one-wayness} the unidirectionality between $\Omega$ and $\Dom U$ in a general setting.
In Section~\ref{T-convergent re} we present theorems which play a crucial role
in establishing the bidirectionality in Section~\ref{two-wayness}.
Based on them, we show in Section~\ref{two-wayness} the bidirectionality between $Z(T)$ and $\Dom U$ with a computable real $T\in(0,1)$ in a general setting.
We conclude this paper with
the remarks on the origin of the phase transition of the behavior of $Z(T)$ between $T=1$ and $T<1$ in Section~\ref{conclusion}.

\section{Preliminaries}
\label{preliminaries}

\subsection{Basic Notation}
\label{basic notation}

We start with some notation about numbers and strings
which will be used in this paper.
$\N=\left\{0,1,2,3,\dotsc\right\}$ is the set of natural numbers,
and $\N^+$ is the set of positive integers.
$\Z$ is the set of integers, and
$\Q$ is the set of rationals.
$\R$ is the set of reals.
%
A sequence $\{a_n\}_{n\in\N}$ of numbers
(rationals or reals)
is called \textit{increasing} if $a_{n+1}>a_{n}$ for all $n\in\N$.

Normally, $o(n)$ denotes any function $f\colon \N^+\to\R$ such that $\lim_{n \to \infty}f(n)/n=0$.
On the other hand, $O(1)$ denotes any function $g\colon \N^+\to\R$ such that there is $C\in\R$ with the property that $\abs{g(n)}\le C$ for all $n\in\N^+$.

$\X=
\left\{
  \lambda,0,1,00,01,10,11,000,001,010,\dotsc
\right\}$
is the set of finite binary strings
where $\lambda$ denotes the \textit{empty string},
and $\X$ is ordered as indicated.
We identify any string in $\X$ with a natural number in this order,
i.e., we consider $\varphi\colon \X\to\N$ such that $\varphi(s)=1s-1$ where
the concatenation $1s$ of strings $1$ and $s$ is regarded as a dyadic integer, and then we identify $s$ with $\varphi(s)$.
For any $s \in \X$, $\abs{s}$ is the \textit{length} of $s$.
For any $n\in\N$, we denote by $\{0,1\}^n$
the set $\{\,s\mid s\in\X\;\&\;\abs{s}=n\}$.
A subset $S$ of $\X$ is called
\textit{prefix-free}
if no string in $S$ is a prefix of another string in $S$.
For any subset $S$ of $\X$ and any $n\in\N$,
we denote by $\rest{S}{n}$
the set $\{s\in S\mid \abs{s}\le n\}$.
For any
function $f$,
the domain of definition of $f$ is denoted by $\Dom f$.
We write ``r.e.'' instead of ``recursively enumerable.''

Let $\alpha$ be an arbitrary real.
$\lfloor \alpha \rfloor$ is the greatest integer less than or equal to $\alpha$,
and $\lceil \alpha \rceil$ is the smallest integer greater than or equal to $\alpha$.
For
any $n\in\N$,
we denote by $\rest{\alpha}{n}\in\X$
the first $n$ bits of the base-two expansion of
$\alpha - \lfloor \alpha \rfloor$ with infinitely many zeros.
For example, in the case of $\alpha=5/8$, $\rest{\alpha}{6}=101000$.
On the other hand, for any non-positive integer $n\in\Z$, we set $\rest{\alpha}{n}=\lambda$.

A real $\alpha$ is called \textit{r.e.}~if
there exists a computable,
increasing sequence of rationals which converges to $\alpha$.
An r.e.~real is also called a
\textit{left-computable} real.
We say that a real $\alpha$ is \textit{computable} if
there exists a computable sequence $\{a_n\}_{n\in\N}$ of rationals
such that $\abs{\alpha-a_n} < 2^{-n}$ for all $n\in\N$.
It is then easy to see that,
for every real $\alpha$,
the following four conditions
are equivalent:
(i) $\alpha$ is computable.
(ii) $\alpha$ is r.e.~and $-\alpha$ is r.e.
(iii) If $f\colon\N\to\Z$ with $f(n)=\lceil\alpha n\rceil$
then $f$ is a total recursive function.
(iv) If $g\colon\N\to\Z$ with $g(n)=\lfloor\alpha n\rfloor$
then $g$ is a total recursive function.

\subsection{Algorithmic Information Theory}
\label{ait}

In the following
we concisely review some definitions and results of
AIT
\cite{C75,C87b,C02,N09,DH10}.
A \textit{prefix-free machine} is a partial recursive function $F\colon \X\to \X$
such that $\Dom F$ is a prefix-free set.
For each prefix-free machine $F$ and each $s \in \X$, $H_F(s)$ is defined by
$H_F(s) =
\min
\left\{\,
  \abs{p}\,\big|\;p \in \X\>\&\>F(p)=s
\,\right\}$
(may be $\infty$).
A prefix-free machine $U$ is said to be \textit{optimal} if
for each prefix-free machine $F$ there exists $d\in\N$
with the following property;
if $p\in\Dom F$, then there is $q\in\Dom U$ for which
$U(q)=F(p)$ and $\abs{q}\le\abs{p}+d$.
It is then easy to see that there exists an optimal prefix-free machine.
We choose a particular optimal prefix-free machine $U$
as the standard one for use,
and define $H(s)$ as $H_U(s)$,
which is referred to as
the \textit{program-size complexity} of $s$, the \textit{information content} of $s$, or the \textit{Kolmogorov complexity} of $s$ \cite{G74,L74,C75}.
%
For any $s,t\in\X$,
we define $H(s,t)$ as $H(b(s,t))$,
where $b\colon \X\times \X\to \X$ is
a particular bijective total recursive function.

Chaitin~\cite{C75} introduced $\Omega$ number as follows.
For each optimal prefix-free machine $V$, the halting probability $\Omega_V$ of $V$ is defined by
\begin{equation*}
  \Omega_V=\sum_{p\in\Dom V}2^{-\abs{p}}.
\end{equation*}
For every optimal prefix-free machine $V$,
since $\Dom V$ is prefix-free,
$\Omega_V$ converges and $0<\Omega_V\le 1$.
For any $\alpha\in\R$,
we say that $\alpha$ is \textit{weakly Chaitin random}
if there exists $c\in\N$ such that
$n-c\le H(\rest{\alpha}{n})$ for all $n\in\N^+$
\cite{C75,C87b}.
Chaitin \cite{C75} showed that
$\Omega_V$ is weakly Chaitin random
for every optimal prefix-free machine $V$.
Therefore $0<\Omega_V<1$ for every optimal prefix-free machine $V$.

Let $M$ be
a
deterministic Turing machine
with the input and output alphabet $\{0,1\}$, 
and let
$F$
be a prefix-free machine.
We say that $M$ \textit{computes} $F$ if the following holds:
for every $p\in\X$,
when $M$ starts with the input $p$,
(i) $M$ halts and outputs $F(p)$ if $p\in\Dom F$;
(ii) $M$ does not halt forever otherwise.
We use this convention on the computation of a prefix-free machine
by a deterministic Turing machine
throughout the rest of this paper.
Thus, we exclude the possibility that
there is $p\in\X$ such that,
when $M$ starts with the input $p$,
$M$ halts
but
$p\notin\Dom F$.
For any $p\in\X$,
we denote the running time of $M$ on the input $p$
by $T_M(p)$ (may be $\infty$).
Thus, 
$T_M(p)\in\N$ for every $p\in\Dom F$
if $M$ computes $F$.

We define
$L_M=
\min\{\,\abs{p}\mid p\in\X\text{ \& $M$ halts on input $p$}\}$
(may be $\infty$).
For any $n\ge L_M$,
we define $I_M^n$
as the set of all halting inputs $p$ for $M$ with $\abs{p}\le n$
which take longest to halt in the computation of $M$,
i.e.,
as the set $\{\,p\in\X\mid \abs{p}\le n\;\&\;T_M(p)=T_M^n\,\}$
where $T_M^n$ is the maximum running time of $M$
on all halting inputs of length at most $n$.
In the work~\cite{T09CiE}, we slightly strengthened the result presented in Chaitin \cite{C87b} to obtain
Theorem~\ref{time} below
(see Note in Section~8.1 of Chaitin \cite{C87b}).
We include the proof of Theorem~\ref{time} in Appendix~\ref{proof-time} since the proof is omitted in the work~\cite{T09CiE}.

\begin{theorem}[Chaitin~\cite{C87b} and Tadaki~\cite{T09CiE}]\label{time}
Let $V$ be an optimal prefix-free machine,
and let $M$ be a deterministic Turing machine which computes $V$.
Then
$n=H(n,p)+O(1)=H(p)+O(1)$
for all $(n,p)$ with $n\ge L_M$ and $p\in I_M^n$.
\qed
\end{theorem}

\subsection{Partial Randomness}
\label{partial}

In the work \cite{T02},
we generalized the notion of the randomness of
a real
so that \textit{the degree of the randomness}, which is often referred to as
\textit{the partial randomness} recently
\cite{CST06,RS05,CS06},
can be characterized by a real $T$ with $0\le T\le 1$ as follows.

\begin{definition}
  Let $T\in[0,1]$ and let $\alpha\in\R$.
  We say that $\alpha$ is \textit{weakly Chaitin $T$-random} if
  there exists $c\in\N$ such that, for all $n\in\N^+$,
  $Tn-c \le H(\rest{\alpha}{n})$.
  \qed
\end{definition}

In the case of $T=1$,
the weak Chaitin $T$-randomness results in the weak Chaitin randomness.

\begin{definition}
Let $T\in[0,1]$ and let $\alpha\in\R$.
We say that $\alpha$ is \textit{$T$-compressible} if
$H(\rest{\alpha}{n})\le Tn+o(n)$,
namely, if
$\limsup_{n \to \infty}H(\rest{\alpha}{n})/n\le T$.
We say that $\alpha$ is \textit{strictly $T$-compressible} if
there exists $d\in\N$ such that, for all $n\in\N^+$,
$H(\rest{\alpha}{n})\le Tn+d$.
\qed
\end{definition}

For every $T\in[0,1]$ and every $\alpha\in\R$,
if $\alpha$ is weakly Chaitin $T$-random and $T$-compressible,
then $\lim_{n\to \infty} H(\rest{\alpha}{n})/n = T$,
i.e., the \textit{compression rate} of $\alpha$ equals to $T$.

In the work \cite{T02},
we generalized
Chaitin $\Omega$ number
to $Z(T)$ as follows.
For each optimal prefix-free machine $V$ and each real $T>0$,
the \textit{partition function} $Z_V(T)$ of $V$ at temperature $T$
is defined by
\begin{equation*}
  Z_V(T) = \sum_{p\in\Dom V}2^{-\frac{\abs{p}}{T}}.
\end{equation*}
Thus,
$Z_V(1)=\Omega_V$.
If $0<T\le 1$, then $Z_V(T)$ converges and $0<Z_V(T)<1$, since $Z_V(T)\le \Omega_V<1$.
The following theorem holds for $Z_V(T)$.

\begin{theorem}[Tadaki \cite{T02}]\label{pomgd}
Let $V$ be an optimal prefix-free machine.
\begin{enumerate}
  \item If $0<T\le 1$ and $T$ is computable, then $Z_V(T)$ is an r.e.~real which is weakly Chaitin $T$-random and $T$-compressible.
  \item If $1<T$, then $Z_V(T)$ diverges to $\infty$.\qed
\end{enumerate}
\end{theorem}

An r.e.~real has a special property on partial randomness, as shown in Theorem~\ref{partial randomness} below.
For any r.e.~reals $\alpha$ and $\beta$, we say that $\alpha$ \textit{dominates} $\beta$
if there are computable, increasing sequences $\{a_n\}$ and $\{b_n\}$ of rationals and $c\in\N^+$ such that
$\lim_{n\to\infty} a_n =\alpha$, $\lim_{n\to\infty} b_n =\beta$, and $c(\alpha-a_n)\ge \beta-b_n$ for all $n\in\N$
\cite{Sol75}.

\begin{definition}%
[Tadaki \cite{T09MFCS}]
Let $T\in(0,1]$.
An increasing sequence $\{a_n\}$ of reals is called
\textit{$T$-convergent} if
$\sum_{n=0}^{\infty} (a_{n+1}-a_{n})^T\!<\infty$.
An r.e.~real $\alpha$ is called \textit{$T$-convergent} if
there exists a $T$-convergent computable,
increasing sequence of rationals which
converges to $\alpha$.
An r.e.~real $\alpha$ is called \textit{$\Omega(T)$-like}
if it dominates all $T$-convergent r.e.~reals.
\qed
\end{definition}

\begin{theorem}[equivalent characterizations of partial randomness for an r.e.~real,
Tadaki \cite{T09MFCS}]\label{partial randomness}
Let $T$ be a computable real in $(0,1]$,
and let $\alpha$ be an r.e.~real.
Then the following three conditions are equivalent:
(i) $\alpha$ is weakly Chaitin $T$-random.
(ii) $\alpha$ is $\Omega(T)$-like.
(iii) For every $T$-convergent r.e.~real $\beta$
there exists $d\in\N$ such that, for all $n\in\N^+$,
$H(\rest{\beta}{n})\le H(\rest{\alpha}{n})+d$.
\qed
\end{theorem}

\section{\boldmath Reducibility in Query Size $f$}
\label{elaboration}

In this section we investigate some properties of the notion of
reducibility
in query size $f$ and introduce the notions of unidirectionality and bidirectionality between two sets.

Note first that,
for every $A\subset\X$,
$A$ is reducible to $A$ in query size $n$,
where ``$n$'' denotes the identity function $I\colon \N\to \N$ with $I(n)=n$.
We follow the notation in computational complexity theory.
%

The following are simple observations on the notion of
reducibility
in query size $f$.

\begin{proposition}\label{observations}
Let $f\colon\N\to\N$ and $g\colon\N\to\N$, and let $A,B,C\subset\X$.
\begin{enumerate}
  \item If $A$ is reducible to $B$ in query size $f$
    and $B$ is reducible to $C$ in query size $g$,
    then $A$ is reducible to $C$ in query size $g\circ f$.
    \label{composition}
  \item Suppose that $f(n)\le g(n)$ for every $n\in\N$.
    If $A$ is reducible to $B$ in query size $f$
    then $A$ is reducible to $B$ in query size $g$.
    \label{longer}
  \item Suppose that $A$ is reducible to $B$ in query size $f$.
    If $A$ is not recursive then $f$ is unbounded.
    \qed
    \label{unbounded}
\end{enumerate}
\end{proposition}

\begin{definition}
An \textit{order function} is
a non-decreasing total recursive function $f\colon\N\to\N$
such that $\lim_{n\to\infty}f(n)=\infty$.
\qed
\end{definition}

Let $f$ be an order function.
Intuitively,
the notion of the reduction of $A$ to $B$ in query size $f$ is
equivalent to that,
for every $n\in\N$,
if $n$ and $\rest{B}{f(n)}$ are given,
then $\rest{A}{n}$ can be
calculated.
We introduce the notions of unidirectionality and bidirectionality between two sets
as follows.

\begin{definition}
Let $A,B\subset\X$.
We say that \textit{the computation from $A$ to $B$ is unidirectional} if the following holds:
For every order functions $f$ and $g$,
if $B$ is reducible to $A$ in query size $f$ and $A$ is reducible to $B$ in query size $g$
then the function $g(f(n))-n$ of $n\in\N$ is unbounded.
We say that \textit{the computations between $A$ and $B$ are bidirectional}
if the computation from $A$ to $B$ is not unidirectional and
the computation from $B$ to $A$ is not unidirectional.
\qed
\end{definition}

The notion of unidirectionality of the computation from $A$ to $B$ in the above definition
is, in essence, interpreted as follows:
No matter how a order function $f$ is chosen,
if $f$ satisfies that
$\rest{B}{n}$ can be calculated from $n$ and $\rest{A}{f(n)}$,
then $\rest{A}{f(n)}$ cannot be calculated from $n$ and $\rest{B}{n+O(1)}$.

%
In order to apply the notion of
reducibility
in query size $f$ to a real, we introduce the notion of prefixes of a real as follows.

\begin{definition}%
For each $\alpha\in\R$,
the \textit{prefixes} $\Pf(\alpha)$ of $\alpha$ is the subset of $\X$
defined by $\Pf(\alpha)=\{\rest{\alpha}{n}\mid n\in\N\}$.
\qed
\end{definition}

The notion of prefixes of a real is a natural notion in AIT.
For example,
the notion of weak Chaitin randomness of a real $\alpha$
can be rephrased as that
there exists $d\in\N$ such that, for every $x\in\Pf(\alpha)$,
$\abs{x}\le H(x)+d$.
The following proposition is a restatement of
the well-known fact that,
for every optimal prefix-free machine $V$,
the first $n$ bits of the base-two expansion of $\Omega_V$ solve
the halting problem of $V$ for inputs of
length at most $n$.

\begin{proposition}\label{DomVrdPfOVqsn}
Let $V$ be an optimal prefix-free machine.
Then $\Dom V$ is reducible to $\Pf(\Omega_V)$ in query size $n$.
\qed
\end{proposition}

\section{Unidirectionality}
\label{one-wayness}

In this section we show the unidirectionality between $\Omega_U$ and $\Dom U$ in a general setting.
Theorems~\ref{one-way I} and \ref{one-way II} below are two of the main results of this paper.

\begin{theorem}[
elaboration of $\Omega_U\le_{wtt}\Dom U$]
\label{one-way I}
Let $V$ and $W$ be optimal prefix-free machines, and
let $f$ be an order function.
Then the following two conditions are equivalent:
\begin{enumerate}
  \item $\Pf(\Omega_V)$ is reducible to $\Dom W$ in query size $f(n)+O(1)$.
  \item $\sum_{n=0}^\infty 2^{n-f(n)}<\infty$.\qed
\end{enumerate}
\end{theorem}

Theorem~\ref{one-way I} is proved in Subsection~\ref{Proof_of_Theorem_one-way I} below.
Theorem~\ref{one-way I} corresponds to Theorem 4 of Tadaki \cite{T09CiE},
and is proved by modifying the proof of Theorem 4 of \cite{T09CiE}.
Let $V$ and $W$ be optimal prefix-free machines.
The implication (ii) $\Rightarrow$ (i) of Theorem~\ref{one-way I} results in,
for example,
that $\Pf(\Omega_V)$ is reducible to $\Dom W$
in query size $n+\lfloor(1+\varepsilon)\log_2 n \rfloor+O(1)$
for every
real $\varepsilon>0$.
On the other hand,
the implication (i) $\Rightarrow$ (ii) of Theorem~\ref{one-way I} results in,
for example,
that $\Pf(\Omega_V)$ is not reducible to $\Dom W$
in query size $n+\lfloor\log_2 n \rfloor+O(1)$
and therefore, in particular,
$\Pf(\Omega_V)$ is not reducible to $\Dom W$ in query size $n+O(1)$.

\begin{theorem}[
elaboration of $\Dom U\le_{wtt}\Omega_U$]
\label{one-way II}
Let $V$ and $W$ be optimal prefix-free machines, and
let $f$ be an order function.
Then the following two conditions are equivalent:
\begin{enumerate}
  \item $\Dom W$ is reducible to $\Pf(\Omega_V)$ in query size $f(n)+O(1)$.
  \item
    $n\le f(n)+O(1)$.
    \qed
\end{enumerate}
\end{theorem}

Theorem~\ref{one-way II} is proved in Subsection~\ref{Proof_of_Theorem_one-way II} below.
Theorem~\ref{one-way II} corresponds to Theorem 11 of Tadaki \cite{T09CiE},
and is proved by modifying the proof of Theorem 11 of \cite{T09CiE}.
The implication (ii) $\Rightarrow$ (i) of Theorem~\ref{one-way II} results in
that, for every optimal prefix-free machines $V$ and $W$,
$\Dom W$ is reducible to $\Pf(\Omega_V)$ in query size $n+O(1)$.
On the other hand,
the implication (i) $\Rightarrow$ (ii) of Theorem~\ref{one-way II} says that
this upper bound ``$n+O(1)$'' of the query size is,
in essence, tight.

\begin{theorem}\label{one-wayness_in_both_directions}
Let $V$ and $W$ be optimal prefix-free machines.
Then the computation from $\Pf(\Omega_V)$ to $\Dom W$ is unidirectional
and the computation from $\Dom W$ to $\Pf(\Omega_V)$ is also unidirectional.
\end{theorem}

\begin{proof}
Let $V$ and $W$ be optimal prefix-free machines.
For arbitrary order functions $f$ and $g$,
assume that $\Dom W$ is reducible to $\Pf(\Omega_V)$ in query size $f$
and $\Pf(\Omega_V)$ is reducible to $\Dom W$ in query size $g$.
It follows from the implication (i) $\Rightarrow$ (ii) of Theorem~\ref{one-way II}
that there exists $c\in\N$ for which $n\le f(n)+c$ for all $n\in\N$.
On the other hand,
it follows from the implication (i) $\Rightarrow$ (ii) of Theorem~\ref{one-way I}
that $\sum_{n=0}^\infty 2^{n-g(n)}<\infty$ and
therefore $\lim_{n\to\infty}g(n)-n=\infty$.
Since $g$ is an order function, we have $g(f(n))-n\ge g(n-c)-(n-c)-c$
for all $n\ge c$.
Thus, the computation from $\Pf(\Omega_V)$ to $\Dom W$ is unidirectional.
On the other hand,
we have $f(g(n))-n\ge g(n)-n-c$ for all $n\in\N$.
Thus,
the computation from $\Dom W$ to $\Pf(\Omega_V)$ is unidirectional.
\end{proof}

\subsection{The Proof of Theorem~\ref{one-way I}}
\label{Proof_of_Theorem_one-way I}

Theorem~\ref{one-way I} follows from
Theorem~\ref{I-re} and Theorem~\ref{I-random} below,
and
the fact that $\Omega_V$ is a weakly Chaitin random r.e.~real
for every optimal prefix-free machine $V$.

\begin{theorem}\label{I-re}
Let $\alpha$ be an r.e.~real,
and let $V$ be an optimal prefix-free machine.
For every total recursive function $f\colon \N\to \N$,
if $\sum_{n=0}^\infty 2^{n-f(n)}<\infty$,
then there exists $c\in\N$ such that
$\Pf(\alpha)$ is reducible to $\Dom V$ in query size $f(n)+c$.
\qed
\end{theorem}

\begin{theorem}\label{I-random}
Let $\alpha$ be a real which is weakly Chaitin random,
and let $V$ be an optimal prefix-free machine.
For every order function $f$,
if $\Pf(\alpha)$ is reducible to $\Dom V$ in query size $f$
then $\sum_{n=0}^\infty 2^{n-f(n)}<\infty$.
\qed
\end{theorem}

We first prove Theorem~\ref{I-re}.
For that purpose,
we need Theorems~\ref{KC} and \ref{op-p} below.

\begin{theorem}[Kraft-Chaitin Theorem, Chaitin \cite{C75}]\label{KC}
Let $f\colon \N\to \N$ be a total recursive function
such that $\sum_{n=0}^\infty 2^{-f(n)}\le 1$.
Then there exists a total recursive function $g\colon \N\to \X$
such that
(i)
$g$ is an injection,
(ii) the set $\{\,g(n)\mid n\in\N\}$ is prefix-free, and
(iii) $\abs{g(n)}=f(n)$ for all $n\in\N$.
\qed
\end{theorem}

We refer to Theorem~\ref{weaksim} below from Tadaki~\cite{T09CiE}.
Theorem~\ref{op-p} is a restatement of
it.

\begin{theorem}[Tadaki \cite{T09CiE}]\label{weaksim}
Let $V$ be an optimal prefix-free machine.
Then, for every prefix-free machine $F$ there exists $d\in\N$ such that,
for every $p\in\X$,
if $p$ and the list of all halting inputs for $V$ of length
at most $\abs{p}+d$ are given,
then the halting problem of the input $p$ for $F$ can be solved.
\qed
\end{theorem}

\begin{theorem}\label{op-p}
Let $V$ be an optimal prefix-free machine.
Then, for every prefix-free machine $F$
there exists $d\in\N$ such that
$\Dom F$ is reducible to $\Dom V$ in query size $n+d$.
\qed
\end{theorem}

Based on Theorems~\ref{KC} and \ref{op-p},
Theorem~\ref{I-re} is then proved as follows.

\begin{proof}[Proof of Theorem~\ref{I-re}]
Let $\alpha$ be an r.e.~real,
and let $V$ be an optimal prefix-free machine.
For an arbitrary total recursive function $f\colon \N\to \N$,
assume that $\sum_{n=0}^\infty 2^{n-f(n)}<\infty$.
In the case of $\alpha\in\Q$,
the result is obvious.
Thus,
in what follows,
we assume that $\alpha\notin\Q$ and therefore
the base-two expansion of
$\alpha - \lfloor \alpha \rfloor$ is unique
and contains infinitely many ones.

Since $\sum_{n=0}^\infty 2^{n-f(n)}<\infty$,
there exists $d_0\in\N$ such that
$\sum_{n=0}^\infty 2^{n-f(n)-d_0}\le 1$.
Hence, by the Kraft-Chaitin Theorem, i.e., Theorem~\ref{KC},
there exists a total recursive function
$g\colon \N\to \X$ such that
(i) the function $g$ is an injection,
(ii) the set $\{\,g(n)\mid n\in\N\}$ is prefix-free, and
(iii) $\abs{g(n)}=f(n)-n+d_0$ for all $n\in\N$.
On the other hand, since $\alpha$ is r.e.,
there exists a total recursive function $h\colon\N\to\Q$ such that
$h(k)\le\alpha$ for all $k\in\N$ and $\lim_{k\to\infty} h(k)=\alpha$.

Now,
let us consider a prefix-free machine $F$ such that,
for every $n\in\N$ and $s\in\X$,
$g(n)s\in\Dom F$ if and only if
(i) $\abs{s}=n$ and
(ii) $0.s<h(k)-\lfloor \alpha \rfloor$ for some $k\in\N$.
It is easy to see that such a prefix-free machine $F$ exists.
We then see that,
for every $n\in\N$ and $s\in\{0,1\}^n$,
\begin{equation}\label{decision-condition}
  g(n)s\in\Dom F
  \text{ if and only if }
  s\le \rest{\alpha}{n},
\end{equation}
where $s$ and $\rest{\alpha}{n}$ are
regarded as a dyadic integer.
Then, by the following procedure,
we see that
$\Pf(\alpha)$ is reducible to $\Dom F$ in query size $f(n)+d_0$.

Given $t\in\X$, based on the equivalence \eqref{decision-condition},
one determines $\rest{\alpha}{n}$
by putting the queries $g(n)s$ to the oracle $\Dom F$
for all $s\in\{0,1\}^n$, where $n=\abs{t}$.
Note here that
all the queries
are of length $f(n)+d_0$,
since $\abs{g(n)}=f(n)-n+d_0$.
One then accepts if $t=\rest{\alpha}{n}$ and rejects otherwise.

On the other hand, by Theorem~\ref{op-p},
there exists $d\in\N$ such that
$\Dom F$ is reducible to $\Dom V$ in query size $n+d$.
Thus, by Proposition~\ref{observations} (i),
$\Pf(\alpha)$ is reducible to $\Dom V$ in query size $f(n)+d_0+d$,
as desired.
\end{proof}

We next prove Theorem~\ref{I-random}.
For that purpose, we need Theorem~\ref{time} and the Ample Excess Lemma
below.

\begin{theorem}[Ample Excess Lemma, Miller and Yu \cite{MY08}]\label{AEL}
For every $\alpha\in\R$,
$\alpha$ is weakly Chaitin random if and only if
$\sum_{n=1}^{\infty} 2^{n-H(\reste{\alpha}{n})}<\infty$.
\qed
\end{theorem}


\begin{proof}[Proof of Theorem~\ref{I-random}]
Let $\alpha$ be a real which is weakly Chaitin random,
and let $V$ be an optimal prefix-free machine.
For an arbitrary order function $f$,
assume that $\Pf(\alpha)$ is reducible to $\Dom V$ in query size $f$.
Since $f$ is an order function,
$S_f=\{n\in\N\mid f(n)<f(n+1)\}$ is an infinite recursive set.
Therefore
there exists an increasing total recursive function $h\colon\N\to\N$
such that $h(\N)=S_f$.
It is then easy to see that
$f(n)=f(h(k+1))$ for every $k$ and $n$ with $h(k)<n\le h(k+1)$.
Thus, for each $k\ge 1$,
we see that
\begin{equation}\label{key_idea}
\begin{split}
  \sum_{n=h(0)+1}^{h(k)}2^{n-f(n)}
  &=\sum_{j=0}^{k-1}\sum_{n=h(j)+1}^{h(j+1)}2^{n-f(n)}
  =\sum_{j=0}^{k-1}2^{-f(h(j+1))}
    \sum_{n=h(j)+1}^{h(j+1)}2^{n}\\
  &=\sum_{j=0}^{k-1}2^{-f(h(j+1))}
    \left(2^{h(j+1)+1}-2^{h(j)+1}\right)
  <2\sum_{j=1}^{k}2^{h(j)-f(h(j))}.
\end{split}
\end{equation}

On the other hand,
let $M$ be a deterministic Turing machine which computes $V$.
For each $n\ge L_M$,
we choose a particular $p_n$ from $I_M^n$.
Note that,
given $(n,p_{f(n)})$ with $f(n)\ge L_M$,
one can calculate the finite set $\rest{\Dom V}{f(n)}$
by simulating the computation of $M$ with the input $q$
until at most the time step $T_M(p_{f(n)})$,
for each $q\in\X$ with $\abs{q}\le f(n)$.
This can be possible because $T_M(p_{f(n)})=T_M^{f(n)}$ for every $n\in\N$ with $f(n)\ge L_M$.
Thus, since $\Pf(\alpha)$ is reducible to $\Dom V$ in query size $f$ by the assumption,
we see that there exists a partial recursive function
$\Psi\colon \N\times\X\to\X$ such that,
for all $n\in\N$ with $f(n)\ge L_M$,
$\Psi(n,p_{f(n)})=\rest{\alpha}{n}$.
It follows
from the optimality of $U$
that
$H(\rest{\alpha}{n})\le H(n,p_{f(n)})+O(1)$
for all $n\in\N$ with $f(n)\ge L_M$.
On the other hand,
since the mapping $\N\ni k\mapsto f(h(k))$ is an increasing total recursive function,
it follows also
from the optimality of $U$
that
$H(h(k),s)\le H(f(h(k)),s)+O(1)$
for all $k\in\N$ and $s\in\X$.
Therefore, using Theorem~\ref{time} we see that
\begin{equation}\label{complexity-bound}
  H(\rest{\alpha}{h(k)})\le f(h(k))+O(1)
\end{equation}
for all $k\in\N$.
Since $\alpha$ is weakly Chaitin random,
using the Ample Excess Lemma, i.e., Theorem~\ref{AEL},
we have $\sum_{n=1}^{\infty} 2^{n-H(\reste{\alpha}{n})}<\infty$.
Note that the function $h$ is
injective.
Thus, using \eqref{complexity-bound} we have
\begin{equation*}
  \sum_{j=1}^{\infty} 2^{h(j)-f(h(j))}
  \le
  \sum_{j=1}^{\infty} 2^{h(j)-H(\reste{\alpha}{h(j)})+O(1)}
  \le
  \sum_{n=1}^{\infty} 2^{n-H(\reste{\alpha}{n})+O(1)}
  <\infty.
\end{equation*}
It follows from \eqref{key_idea} that
$\lim_{k\to\infty}\sum_{n=h(0)+1}^{h(k)}2^{n-f(n)}<\infty$.
Thus, since $2^{n-f(n)}>0$ for all $n\in\N$
and $\lim_{k\to\infty}h(k)=\infty$,
we have $\sum_{n=0}^{\infty}2^{n-f(n)}<\infty$, as desired.
\end{proof}

\subsection{The Proof of Theorem~\ref{one-way II}}
\label{Proof_of_Theorem_one-way II}

The implication (ii) $\Rightarrow$ (i) of Theorem~\ref{one-way II} follows immediately from
Proposition~\ref{DomVrdPfOVqsn} and Proposition~\ref{observations} (ii).
On the other hand, the implication (i) $\Rightarrow$ (ii) of Theorem~\ref{one-way II} is proved as follows.

\begin{proof}[Proof of (i) $\Rightarrow$ (ii) of Theorem~\ref{one-way II}]
Let $V$ and $W$ be optimal prefix-free machines, and let $f$ be an order function.
Suppose that there exists $c\in\N$ such that $\Dom W$ is reducible to $\Pf(\Omega_V)$ in query size $f(n)+c$.
Then, by considering the following procedure,
we first see that $n<H(n,\rest{\Omega_V}{f(n)+c})+O(1)$
for all $n\in\N$.

Given $n$ and $\rest{\Omega_V}{f(n)+c}$,
one first calculates the finite set $\rest{\Dom W}{n}$.
This is possible since
$\Dom W$ is reducible to $\Pf(\Omega_V)$ in query size $f(n)+c$
and $f(k)\le f(n)$ for all $k\le n$.
Then,
by calculating the set $\{\,W(p)\mid p\in\rest{\Dom W}{n}\}$
and picking any one finite binary string $s$ which is not in this set,
one can obtain $s\in\X$ such that $n<H_W(s)$.

Thus, there exists a partial recursive function
$\Psi\colon\N\times\X\to\X$ such that,
for all $n\in\N$,
$n<H_W(\Psi(n,\rest{\Omega_V}{f(n)+c}))$.
It follows
from the optimality of $W$
that
\begin{equation}\label{n-f(n)-lower-bound}
  n<H(n,\rest{\Omega_V}{f(n)+c})+O(1)
\end{equation}
for all $n\in\N$.

Now, let us assume contrarily that
the function $n-f(n)$ of $n\in\N$ is unbounded.
Recall that $f$ is an order function.
Hence it is easy to show that there exists a total recursive function $g\colon\N\to\N$ such that
the function $f(g(k))$ of $k$ is increasing and the function $g(k)-f(g(k))$ of $k$ is also increasing.
For clarity,
we define a total recursive function $m\colon \N\to\N$
by $m(k)=f(g(k))+c$.
Since $m$ is
injective,
it is then easy to see that
there exists a partial recursive function
$\Phi\colon\N\to\N$ such that
$\Phi(m(k))=g(k)$ for all $k\in\N$.
Therefore,
based on the optimality of $U$, 
it is shown that
$H(g(k),\rest{\Omega_V}{m(k)})
\le H(\rest{\Omega_V}{m(k)})+O(1)$
for all $k\in\N$.
It follows from \eqref{n-f(n)-lower-bound} that
$g(k)<H(\rest{\Omega_V}{m(k)})+O(1)$
for all $k\in\N$.
On the other hand,
we can show that $H(s)\le\abs{s}+H(\abs{s})+O(1)$ for all $s\in\X$.
Therefore we have $g(k)-f(g(k))<H(m(k))+O(1)$ for all $k\in\N$.
Then, since the function $g(k)-f(g(k))$ of $k$ is unbounded,
it is easy to see that there exists a total recursive function $\Theta\colon\N^+\to\N$ such that, for every $l\in\N^+$, $l\le H(\Theta(l))$.
It follows
from the optimality of $U$
that $l\le H(l)+O(1)$ for all $l\in\N^+$. 
On the other hand,
we can show that $H(l)\le 2\log_2 l+O(1)$ for all $l\in\N^+$.
Thus we have $l\le 2\log_2 l+O(1)$ for all $l\in\N^+$.
However, we have a contradiction on letting $l\to\infty$ in this inequality.
This completes the proof.
\end{proof}

\section{\boldmath $T$-Convergent R.E.~Reals}
\label{T-convergent re}

Let $T$ be an arbitrary computable real with $0<T\le 1$.
The parameter $T$ plays a crucial role in the present paper.%
\footnote{
The parameter $T$ corresponds to the notion of ``temperature''
in the statistical mechanical interpretation of
AIT
introduced
by Tadaki
\cite{T08CiE}.
}
In this section, we investigate the relation of $T$-convergent r.e.~reals to the halting problems.
In particular, Theorem~\ref{DomVTn-reT} below is used to show
Theorem~\ref{main1}
in the next section, and plays a major role in establishing the bidirectionality in the next section.
On the other hand, Theorem~\ref{ZVTwCTr-strictTcb} below is used to show
Theorem~\ref{main2}
in the next section.

Recently, Calude, Hay, and Stephan \cite{CHS11} showed the existence of an r.e.~real
which is weakly Chaitin $T$-random and strictly $T$-compressible,
in the case where $T$ is a computable real with $0<T<1$,
as follows.

\begin{theorem}[Calude, Hay, and Stephan \cite{CHS11}]\label{CHS}
Suppose that $T$ is a computable real with $0<T<1$.
Then there exist an r.e.~real $\alpha\in(0,1)$ and $d\in\N$ such that,
for all $n\in\N^+$,
$\abs{H(\rest{\alpha}{n})-Tn}\le d$.
\qed
\end{theorem}

We first show that the same r.e.~real $\alpha$ as in Theorem~\ref{CHS} has the following property.

\begin{theorem}\label{DomVTn-arewCTr}
Suppose that $T$ is a computable real with $0<T<1$.
Let $V$ be an optimal prefix-free machine.
Then there exists an r.e.~real $\alpha\in(0,1)$
such that $\alpha$ is weakly Chaitin $T$-random and
$\Pf(\alpha)$ is reducible to $\Dom V$ in query size $\lfloor Tn\rfloor+O(1)$.
\qed
\end{theorem}

Calude, et al.~\cite{CHS11} use Lemma~\ref{RS-CHS} below to show Theorem~\ref{CHS}.
We also use
it
to show
Theorem~\ref{DomVTn-arewCTr}.
We include the proof of Lemma~\ref{RS-CHS} in Appendix~\ref{proof-RS-CHS} for completeness.

\begin{lemma}[Reimann and Stephan~\cite{RS05} and Calude, Hay, and Stephan~\cite{CHS11}]\label{RS-CHS}
Let $T$ be a real with $T>0$, and let $V$ be an optimal prefix-free machine.
\begin{enumerate}
\item Suppose that $T<1$.
  Then there exists $c\in\N^+$ such that, for every $s\in\X$,
  there exists $t\in\{0,1\}^c$ for which $H_V(st)\ge H_V(s)+Tc$.
\item There exists $c\in\N^+$ such that, for every $s\in\X$,
  $H_V(s0^c)\le H_V(s)+Tc-1$ and $H_V(s1^c)\le H_V(s)+Tc-1$.\qed
\end{enumerate}
\end{lemma}

The proof of Theorem~\ref{DomVTn-arewCTr} is then given as follows.

\begin{proof}[Proof of Theorem~\ref{DomVTn-arewCTr}]
Suppose that $T$ is a computable real with $0<T<1$.
Let $V$ be an optimal prefix-free machine.
Then it follows from Lemma~\ref{RS-CHS} that
there exists $c\in\N^+$ such that, for every $s\in\X$,
there exists $t\in\{0,1\}^c$ for which
\begin{equation}\label{HVstgHVspTc}
  H_V(st)\ge H_V(s)+Tc.
\end{equation}
For each prefix-free machine $G$ and each $s\in\X$,
we denote by $S(G;s)$ the set
\begin{equation*}
  \bigl\{\,u\in\{0,1\}^{\abs{s}+c}\bigm|
  \text{$s$ is a prefix of $u$ }\&\;H_G(u)>T\abs{u}\,\bigr\}.
\end{equation*}


Now,
we define a sequence $\{a_k\}_{k\in\N}$ of finite binary strings
recursively on $k\in\N$
by $a_k:=\lambda$ if $k=0$ and $a_k:=\min S(V;a_{k-1})$ otherwise.
First note that $a_0$ is properly defined as $\lambda$ and therefore
satisfies $H_V(a_0)>T\abs{a_0}$.
For each $k\ge 1$,
assume that $a_0,a_1,a_2,\dots,a_{k-1}$ are properly defined.
Then $H_V(a_{k-1})>T\abs{a_{k-1}}$ holds.
It follow from \eqref{HVstgHVspTc} that
there exists $t\in\{0,1\}^c$ for which $H_V(a_{k-1}t)\ge H_V(a_{k-1})+Tc$,
and therefore
$a_{k-1}t\in\{0,1\}^{\abs{a_{k-1}}+c}$ and $H_V(a_{k-1}t)\ge T\abs{a_{k-1}t}$.
Thus $S(V;a_{k-1})\neq\emptyset$,
and therefore $a_{k}$ is properly defined.
Hence, $a_{k}$ is properly defined for every $k\in\N$.
We thus see that, for every $k\in\N$,
$a_{k}\in\{0,1\}^{ck}$, $H_V(a_{k})>T\abs{a_{k}}$, and $a_{k}$ is a prefix of $a_{k+1}$. 
Therefore,
it is easy to see that,
for every $m\in\N^+$,
there exists $k\in\N$ such that $a_{k}$ contains $m$ zeros.
Thus, we can uniquely define a real $\alpha\in[0,1)$ by the condition that
$\rest{\alpha}{ck}=a_{k}$ for all $k\in\N^+$.
It follows that $H_V(\rest{\alpha}{ck})>T\abs{\rest{\alpha}{ck}}$ for all $k\in\N^+$.
Note that there exists $d_0\in\N$ such that, for every $s,t\in\X$,
if $\abs{t}\le c$ then $\abs{H_V(st)-H_V(s)}\le d_0$.
Therefore,
there exists $d_1\in\N$ such that, for every $n\in\N^+$,
$H_V(\rest{\alpha}{n})>Tn-d_1$,
which implies that $\alpha$ is weakly Chaitin $T$-random
and therefore $\alpha\in(0,1)$.

Next, we show that
$\Pf(\alpha)$ is reducible to $\Dom V$ in query size $\lceil Tn\rceil+O(1)$.
For each $k\in\N$,
we denote by $F_k$ the set $\{s\in\X\mid H_V(s)\le \lfloor Tck\rfloor\}$.
It follows that
\begin{equation}\label{ak=ak-1fk}
  a_{k}=
  \min\bigl\{\,u\in\{0,1\}^{ck}\bigm|
  \text{$a_{k-1}$ is a prefix of $u$ }\&\;u\notin F_{k}\,\}
\end{equation}
for every $k\in\N^+$.
By the following procedure, we see that
$\Pf(\alpha)$ is reducible to $\Dom V$ in query size $\lfloor Tn\rfloor+O(1)$.

Given $s\in\X$ with $s\neq\lambda$,
one first calculates
the $k_0$ finite sets $F_1,F_2,\dots,F_{k_0}$, where $k_0=\lceil\abs{s}/c\rceil$,
by putting queries to the oracle $\Dom V$.
Note here that
all the queries can be of length at most $\lfloor T(\abs{s}+c)\rfloor$.
One then calculates $a_1,a_2,\dots,a_{k_0}$ in this order one by one
from $a_0=\lambda$
based on the relation \eqref{ak=ak-1fk} and $F_1,F_2,\dots,F_{k_0}$.
Finally, one accepts $s$ if $s$ is a prefix of $a_{k_0}$ and rejects otherwise.
This is possible since $\rest{\alpha}{ck_0}=a_{k_0}$ and $\abs{s}\le ck_0$.

Finally,
we show that $\alpha$ is an r.e.~real.
Let $p_1,p_2,p_3, \dotsc$ be a particular recursive enumeration of
the infinite r.e.~set $\Dom V$.
For each $l\in\N^+$,
we define a prefix-free machine $V^{(l)}$
by the following two conditions (i) and (ii):
(i) $\Dom V^{(l)}=\{p_1,p_2,\dots,p_l\}$.
(ii)$V^{(l)}(p)=V(p)$ for every $p\in\Dom V^{(l)}$.
It is easy to see that
such prefix-free machines $V^{(1)},V^{(2)},V^{(3)},\dotsc$ exist.
For each $l\in\N^+$ and each $s\in\X$,
note that $H_{V^{(l)}}(s)\ge H_V(s)$ holds,
where $H_{V^{(l)}}(s)$ may be $\infty$.
For each $l\in\N$,
we define a sequence $\{a^{(l)}_k\}_{k\in\N}$ of finite binary strings
recursively on $k\in\N$
by $a^{(l)}_k:=\lambda$ if $k=0$ and
$a^{(l)}_k:=\min (S(V^{(l)};a^{(l)}_{k-1})\cup\{a^{(l)}_{k-1}1^c\})$ otherwise.
It follows that $a^{(l)}_{k}$ is properly defined for every $k\in\N$.
Note, in particular, that
$a^{(l)}_{k}\in\{0,1\}^{ck}$ and $a^{(l)}_{k}$ is a prefix of $a^{(l)}_{k+1}$
for every $k\in\N$.

Let $l\in\N^+$.
We show that $a^{(l)}_k\le a_k$ for every $k\in\N^+$.
To see this,
assume that $a^{(l)}_{k-1}=a_{k-1}$.
Then,
since $H_{V^{(l)}}(s)\ge H_V(s)$ holds for every $s\in\X$,
based on the constructions of
$a^{(l)}_k$ and $a_k$ from $a^{(l)}_{k-1}$ and $a_{k-1}$, respectively,
we see that $a^{(l)}_{k}\le a_{k}$.
Thus,
based on the constructions of
$\{a^{(l)}_k\}_{k\in\N}$ and $\{a_k\}_{k\in\N}$
we see that $a^{(l)}_k\le a_k$ for every $k\in\N^+$.

We define a sequence $\{r_k\}_{k\in\N}$ of rationals by $r_k=0.a^{(k)}_k$.
Obviously,
$\{r_k\}_{k\in\N}$ is a computable sequence of rationals. 
Based on the result in the previous paragraph,
we see that $r_k\le \alpha$ for every $k\in\N^+$.
Based on the constructions of
prefix-free machines $V^{(1)},V^{(2)},V^{(3)},\dotsc$
from $V$,
it is also easy to see that $\lim_{k\to\infty}r_k=\alpha$.
Thus we see that $\alpha$ is an r.e.~real.
\end{proof}

Note that,
using Theorem~\ref{time} and Theorem~\ref{DomVTn-arewCTr},
we can give to Theorem~\ref{CHS} a different proof from Calude, et al.~\cite{CHS11} as follows.

\begin{proof}[Different Proof of Theorem~\ref{CHS} from Calude, et al.~\cite{CHS11}]
Suppose that $T$ is a computable real with $0<T<1$.
We choose a particular optimal prefix-free machine $V$ and a particular deterministic Turing machine $M$ such that $M$ computes $V$.
For each $n$ with $\lceil Tn\rceil\ge L_M$,
we choose a particular $p_n$ from $I_M^{\lceil Tn\rceil}$.
By Theorem~\ref{DomVTn-arewCTr}, there exist an r.e.~real $\alpha\in(0,1)$, an oracle deterministic Turing machine $M_0$, and $c\in\N$ such that
$\alpha$ is weakly Chaitin $T$-random and, for all $n\in\N^+$, $M_0^{\reste{\Dom V}{\lceil Tn\rceil}}(n)=\rest{\alpha}{n-c}$.
Then, by the following procedure, we see that there exists a partial recursive function $\Psi\colon \N\times\X\to\X$ such that, 
for all $n$ with $\lceil Tn\rceil\ge L_M$,
\begin{equation}\label{pnpn=an-c}
  \Psi(n,p_n)=\rest{\alpha}{n-c}.
\end{equation}

Given $(n,p_n)$ with $\lceil Tn\rceil\ge L_M$,
one first calculates the finite set $\rest{\Dom V}{\lceil Tn\rceil}$
by simulating the computation of $M$ with the input $q$
until at most the time step $T_M(p_n)$,
for each $q\in\X$ with $\abs{q}\le \lceil Tn\rceil$.
This can be possible because $T_M(p_n)=T_M^{\lceil Tn\rceil}$
for every $n$ with $\lceil Tn\rceil\ge L_M$.
One then calculates $\rest{\alpha}{n-c}$
by simulating the computation of $M_0$
with the input $n$ and the oracle $\rest{\Dom V}{\lceil Tn\rceil}$.

It follows from \eqref{pnpn=an-c}
that
\begin{equation}\label{Han-clHnp}
  H(\rest{\alpha}{n-c})\le H(n,p_n)+O(1)
\end{equation}
for all $n$ with $\lceil Tn\rceil\ge L_M$.

On the other hand,
given $\lceil Tn\rceil$ with $n\in\N^+$,
one only need to specify one of $\lceil 1/T\rceil$ possibilities
of $n$ in order to calculate $n$,
since $T$ is a computable real and $T\neq 0$.
Thus, there exists a partial recursive function
$\Phi\colon \N^+\times\X\times\N^+\to\N^+\times \X$ such that,
for every $n\in\N^+$ and every $p\in\X$,
there exists $k\in\N^+$ with the properties that
$1\le k\le \lceil 1/T\rceil$ and $\Phi(\lceil Tn\rceil,p,k)=(n,p)$.
It follows
that
$H(n,p)\le
H(\lceil Tn\rceil,p)+
\max\{H(k)\mid k\in\N^+\ \&\ 1\le k\le \lceil 1/T\rceil\,\}+O(1)$
for all $n\in\N^+$ and all $p\in\X$.
Hence,
using \eqref{Han-clHnp} and Theorem~\ref{time} we have
\begin{equation*}
  H(\rest{\alpha}{n-c})\le H(\lceil Tn\rceil,p_n)+O(1)\le \lceil Tn\rceil +O(1)\le Tn+O(1)
\end{equation*}
for all $n$ with $\lceil Tn\rceil\ge L_M$.
It follows that $H(\rest{\alpha}{n})\le Tn+O(1)$ for all $n\in\N^+$,
which implies that $\alpha$ is strictly $T$-compressible.
This completes the proof.
\end{proof}

Using Theorem~\ref{partial randomness} and Theorem~\ref{CHS} we can prove the following theorem.

\begin{theorem}\label{icalp10-uc10}
Suppose that $T$ is a computable real with $0<T<1$.
For every r.e.~real $\beta$,
if $\beta$ is $T$-convergent then $\beta$ is strictly $T$-compressible.
\end{theorem}

\begin{proof}
Suppose that $T$ is a computable real with $0<T<1$.
It follows from Theorem~\ref{CHS} that there exists an r.e.~real $\alpha$ such that $\alpha$ is weakly Chaitin $T$-random and
\begin{equation}\label{HbnlTno}
  H(\rest{\alpha}{n})\le Tn+O(1)
\end{equation}
for all $n\in\N^+$.
Since $\alpha$ is weakly Chaitin $T$-random, using the implication (i) $\Rightarrow$ (iii) of Theorem~\ref{partial randomness} we see that,
for every $T$-convergent r.e.~real $\beta$, there exists $d\in\N$ such that, for all $n\in\N^+$, $H(\rest{\beta}{n})\le H(\rest{\alpha}{n})+d$.
Thus, for each $T$-convergent r.e.~real $\beta$, using \eqref{HbnlTno} we see that $H(\rest{\beta}{n})\le Tn+O(1)$ for all $n\in\N^+$,
which implies that $\beta$ is strictly $T$-compressible.
\end{proof}

Using Theorem 7 of Tadaki~\cite{T09MFCS}, Theorem~\ref{icalp10-uc10}, and Theorem~\ref{pomgd} (i),
we can prove the following theorem.

\begin{theorem}\label{ZVTwCTr-strictTcb}
Suppose that $T$ is a computable real with $0<T<1$. Let $V$ be an optimal prefix-free machine.
Then there exists $d\in\N$ such that,
for all $n\in\N^+$,
$\abs{H(\rest{Z_V(T)}{n})-Tn}\le d$.
\end{theorem}

\begin{proof}
Suppose that $T$ is a computable real with $0<T<1$. Let $V$ be an optimal prefix-free machine.
By Theorem 7 of Tadaki~\cite{T09MFCS},
$Z_V(T)$ is a $T$-convergent r.e.~real.
It follows from Theorem~\ref{icalp10-uc10} that
$Z_V(T)$ is strictly $T$-compressible.
On the other hand,
by Theorem~\ref{pomgd} (i),
$Z_V(T)$ is weakly Chaitin $T$-random.
This completes the proof.
\end{proof}

Calude, et al.~\cite{CHS11}, in essence, showed the following result.
For completeness,
we include its proof.

\begin{theorem}[Calude, Hay, and Stephan \cite{CHS11}]\label{bounded-run}
If a real $\beta$ is weakly Chaitin $T$-random and strictly $T$-compressible,
then
there exists $d\ge 2$ such that a base-two expansion of $\beta$ has neither a run of $d$ consecutive zeros nor a run of $d$ consecutive ones.
\end{theorem}

\begin{proof}
Let $\beta$ be a real which is weakly Chaitin $T$-random and strictly $T$-compressible.
Then there exists $d_0\in\N$ such that, for every $n\in\N$,
\begin{equation}\label{Hbn-Tnld0}
  \abs{H(\rest{\beta}{n})-Tn}\le d_0.
\end{equation}
On the other hand, by Lemma~\ref{RS-CHS} (ii) we see that there exists $c\in\N^+$ such that, for every $s\in\X$, $H(s0^c)\le H(s)+Tc-1$ and $H(s1^c)\le H(s)+Tc-1$.
We choose a particular $k_0\in\N^+$ with $k_0>2d$.

Assume first that
a base-two expansion of $\beta$
has a run of $ck_0$ consecutive zeros.
Then $\rest{\beta}{n_0}0^{ck_0}=\rest{\beta}{n_0+ck_0}$ for some $n_0\in\N$.
Thus we have
$H(\rest{\beta}{n_0+ck_0})-T(n_0+ck_0)+k_0\le H(\rest{\beta}{n_0})-Tn_0$,
and therefore
$-\abs{H(\rest{\beta}{n_0+ck_0})-T(n_0+ck_0)}+k_0\le \abs{H(\rest{\beta}{n_0})-Tn_0}$
where we used the triangle inequality.
It follows from \eqref{Hbn-Tnld0} that $-d_0+k_0\le d_0$ and therefore $k_0\le 2d_0$.
This contradicts the fact that $k_0>2d$.
Hence,
a base-two expansion of $\beta$
does not have a run of $ck_0$ consecutive zeros.
In a similar manner we can show that
a base-two expansion of $\beta$
does not have a run of $ck_0$ consecutive ones, as well.
\end{proof}

\begin{theorem}\label{DomVTn-reT}
Suppose that $T$ is a computable real with $0<T<1$. Let $V$ be an optimal prefix-free machine.
For every r.e.~real $\beta$,
if $\beta$ is $T$-convergent and weakly Chaitin $T$-random,
then $\Pf(\beta)$ is reducible to $\Dom V$ in query size $\lfloor Tn\rfloor+O(1)$.
\end{theorem}

\begin{proof}
Suppose that $T$ is a computable real with $0<T<1$. Let $V$ be an optimal prefix-free machine.
Then, by Theorem~\ref{DomVTn-arewCTr},
there exist an r.e.~real $\alpha\in(0,1)$ and $d_0\in\N$
such that $\alpha$ is weakly Chaitin $T$-random and
$\Pf(\alpha)$ is reducible to $\Dom V$ in query size $\lfloor Tn\rfloor+d_0$.
Since $\alpha$ is an r.e.~real which is weakly Chaitin $T$-random,
it follow from
the implication (i) $\Rightarrow$ (ii) of Theorem~\ref{partial randomness}
that $\alpha$ is $\Omega(T)$-like.

Now, for an arbitrary r.e.~real $\beta$,
assume that $\beta$ is $T$-convergent and weakly Chaitin $T$-random.
Then, by Theorem~\ref{icalp10-uc10},
$\beta$ is strictly $T$-compressible.
It follows from Theorem~\ref{bounded-run} that
there exists $c\ge 2$ such that
the base-two expansion of $\beta$
has neither a run of $c$ consecutive zeros nor a run of $c$ consecutive ones.
On the other hand,
since the r.e.~real $\alpha$ is weakly Chaitin $T$-random,
from the definition of $\Omega(T)$-likeness
we see that $\alpha$ dominates $\beta$.
Therefore,
there are computable, increasing sequences $\{a_k\}_{k\in\N}$ and $\{b_k\}_{k\in\N}$
of rationals and $d_1\in\N$
such that
$\lim_{k\to\infty} a_k =\alpha$ and $\lim_{k\to\infty} b_k =\beta$
and, for all $k\in\N$,
$\alpha-a_k\ge 2^{-d_1}(\beta-b_k)$ and $\lfloor\beta\rfloor=\lfloor b_k\rfloor$.
Let $d_2=d_1+c+2$.
Then, by the following procedure, we see that
$\Pf(\beta)$ is reducible to $\Dom V$ in query size $\lfloor T(n+d_2)\rfloor+d_0$.

Given $s\in\X$,
one first calculates $\rest{\alpha}{n+d_2}$
by putting the queries $t$ to the oracle $\Dom V$,
where $n=\abs{s}$.
This is possible since 
$\Pf(\alpha)$ is reducible to $\Dom V$ in query size $\lfloor Tn\rfloor+d_0$.
Note here that
all the queries can be of length at most $\lfloor T(n+d_2)\rfloor+d_0$.
One then find $k_0\in\N$ such that
$0.(\rest{\alpha}{n+d_2})<a_{k_0}$.
This is possible since
$0.(\rest{\alpha}{n+d_2})<\alpha$ and
$\lim_{k\to\infty}a_{k}=\alpha$.
It follows that
$2^{-(n+d_2)}>\alpha-0.(\rest{\alpha}{n+d_2})>\alpha-a_{k_0}\ge 2^{-d_1}(\beta-b_{k_0})$.
Thus, $0<\beta-b_{k_0}<2^{-(n+c+2)}$.
Let $t$ be the first $n+c+2$ bits of the base-two expansion of
the rational number $b_{k_0}-\lfloor b_{k_0}\rfloor$ with infinitely many zeros.
Then, $\abs{\,b_{k_0}-\lfloor b_{k_0}\rfloor-0.t\,}\le 2^{-(n+c+2)}$.
It follows from $\abs{\,\beta-\lfloor\beta\rfloor-0.(\rest{\beta}{n+c+2})\,}<2^{-(n+c+2)}$
that $\abs{\,0.(\rest{\beta}{n+c+2})-0.t_n\,}<3\cdot 2^{-(n+c+2)}<2^{-(n+c)}$.
Hence, $\abs{\,\rest{\beta}{n+c+2}-t\,}<2^{2}$,
where $\rest{\beta}{n+c+2}$ and $t$ in $\{0,1\}^{n+c+2}$ are regarded as a dyadic integer.
Thus,
$t$ is obtained by adding to $\rest{\beta}{n+c+2}$ or subtracting from $\rest{\beta}{n+c+2}$
a $2$ bits dyadic integer.
Since the base-two expansion of $\beta$
has neither a run of $c$ consecutive zeros nor a run of $c$ consecutive ones,
it can be checked that
the first $n$ bits of $t$ equals to $\rest{\beta}{n}$.
Thus, one accepts $s$ if $s$ is a prefix of $t$ and rejects otherwise.
Recall here that $\abs{s}=n$.
\end{proof}

\section{Bidirectionality}
\label{two-wayness}

In this section we show the bidirectionality between $Z_U(T)$ and $\Dom U$ with a computable real $T\in(0,1)$ in a general setting.
Theorems~\ref{main1} and \ref{main2} below are two of the main results of this paper.

\begin{theorem}[
elaboration of $Z_U(T)\le_{wtt}\Dom U$]\label{main1}
Suppose that $T$ is a computable real with $0<T<1$.
Let $V$ and $W$ be optimal prefix-free machines, and
let $f$ be an order function.
Then the following two conditions are equivalent:
\begin{enumerate}
  \item $\Pf(Z_V(T))$ is reducible to $\Dom W$ in query size $f(n)+O(1)$.
  \item
    $Tn\le f(n)+O(1)$.
    \qed
\end{enumerate}
\end{theorem}

\begin{theorem}[
elaboration of $\Dom U\le_{wtt}Z_U(T)$]\label{main2}
Suppose that $T$ is a computable real with $0<T\le 1$.
Let $V$ and $W$ be optimal prefix-free machines, and
let $f$ be an order function.
Then the following two conditions are equivalent:
\begin{enumerate}
  \item $\Dom W$ is reducible to $\Pf(Z_V(T))$ in query size $f(n)+O(1)$.
  \item
    $n/T\le f(n)+O(1)$.
    \qed
\end{enumerate}
\end{theorem}

Theorem~\ref{main1} and Theorem~\ref{main2} are proved
in Subsection~\ref{Proof_of_Theorem_main1} and Subsection~\ref{Proof_of_Theorem_main2} below, respectively.
%
Note that the function $Tn$ in the condition (ii) of Theorem~\ref{main1} and
the function $n/T$ in the condition (ii) of Theorem~\ref{main2} are
the inverse functions of each other.
This implies that the computations between $\Pf(Z_V(T))$ and $\Dom W$ are bidirectional
in the case where $T$ is a computable real with $0<T<1$.
The formal proof is as follows.

\begin{theorem}\label{two-wayness_for_T<1}
Suppose that $T$ is a computable real with $0<T<1$.
Let $V$ and $W$ be optimal prefix-free machines.
Then the computations between $\Pf(Z_V(T))$ and $\Dom W$ are bidirectional.
\end{theorem}

\begin{proof}
Let $V$ and $W$ be optimal prefix-free machines.
It follows from the implication (ii) $\Rightarrow$ (i) of Theorem~\ref{main2}
that there exists $c\in\N$ for which
$\Dom W$ is reducible to $\Pf(Z_V(T))$ in query size $f$
with $f(n)=\lfloor n/T\rfloor + c$.
On the other hand,
it follows from the implication (ii) $\Rightarrow$ (i) of Theorem~\ref{main1}
that there exists $d\in\N$ for which
$\Pf(Z_V(T))$ is reducible to $\Dom W$ in query size $g$
with $g(n)=\lfloor Tn\rfloor + d$.
Since $T$ is computable, $f$ and $g$ are order functions.
For each $n\in\N$, we see that $g(f(n))\le Tf(n)+d\le n+Tc+d$.
Thus, the computation from $\Pf(\Omega_V)$ to $\Dom W$ is not unidirectional.
In a similar manner,
we see that
the computation from $\Dom W$ to $\Pf(\Omega_V)$ is not unidirectional.
This completes the proof.
\end{proof}

\subsection{The Proof of Theorem~\ref{main1}}
\label{Proof_of_Theorem_main1}

Let $T$ be a computable real with $0<T<1$, and let $V$ be an optimal prefix-free machine.
Then, by Theorem 7 of Tadaki~\cite{T09MFCS}, $Z_V(T)$ is a $T$-convergent r.e.~real.
Moreover, by Theorem~\ref{pomgd}~(i), $Z_V(T)$ is weakly Chaitin $T$-random.
Thus, the implication (ii) $\Rightarrow$ (i) of Theorem~\ref{main1} follows immediately from Theorem~\ref{DomVTn-reT} and Proposition~\ref{observations} (ii).

On the other hand, the implication (i) $\Rightarrow$ (ii) of Theorem~\ref{main1} follows immediately from Theorem~\ref{pomgd}~(i) and Theorem~\ref{I_T-random} below.
In order to prove Theorem~\ref{I_T-random}, we use Theorem~\ref{time}.

\begin{theorem}\label{I_T-random}
Suppose that $T$ is a computable real with $0<T\le 1$.
Let $\beta$ be a real which is weakly Chaitin $T$-random, and let $V$ be an optimal prefix-free machine.
For every order function $f$,
if $\Pf(\beta)$ is reducible to $\Dom V$ in query size $f$
then $Tn\le f(n)+O(1)$.
\end{theorem}

\begin{proof}
Suppose that $T$ is a computable real with $0<T\le 1$.
Let $\beta$ be a real which is weakly Chaitin $T$-random,
and let $V$ be an optimal prefix-free machine.
For an arbitrary order function $f$,
assume that $\Pf(\beta)$ is reducible to $\Dom V$ in query size $f$.
Let $M$ be a deterministic Turing machine which computes $V$.
For each $n$ with $f(n)\ge L_M$,
we choose a particular $p_n$ from $I_M^{f(n)}$.
Then, by the following procedure,
we see that
there exists a partial recursive function
$\Psi\colon \N\times\X\to\X$ such that,
for all $n$ with $f(n)\ge L_M$,
\begin{equation}\label{pnpn=n+fn}
  \Psi(n,p_n)=\rest{\beta}{n}.
\end{equation}

Given $(n,p_n)$ with $f(n)\ge L_M$,
one first calculates the finite set $\rest{\Dom V}{f(n)}$
by simulating the computation of $M$ with the input $q$
until at most the time step $T_M(p_n)$,
for each $q\in\X$ with $\abs{q}\le f(n)$.
This can be possible because $T_M(p_n)=T_M^{f(n)}$
for every $n$ with $f(n)\ge L_M$. 
One then calculates $\rest{\beta}{n}$
using $\rest{\Dom V}{f(n)}$ and outputs it.
This is possible since $\Pf(\beta)$ is reducible to $\Dom V$ in query size $f$.

It follows from \eqref{pnpn=n+fn}
that
\begin{equation}\label{HbnpflHnp}
  H(\rest{\beta}{n})\le H(n,p_n)+O(1)
\end{equation}
for all $n$ with $f(n)\ge L_M$.

Now, let us assume contrarily that the function $Tn-f(n)$ of $n\in\N$ is unbounded.
Recall that $f$ is an order function and $T$ is computable.
Hence it is easy to show that there exists a total recursive function $g\colon\N\to\N$ such that
the function $f(g(k))$ of $k$ is increasing and the function $Tg(k)-f(g(k))$ of $k$ is also increasing.
Since the function $f(g(k))$ of $k$ is injective, it is then easy to see that there exists a partial recursive function
$\Phi\colon\N\to\N$ such that
$\Phi(f(g(k)))=g(k)$ for all $k\in\N$.
Thus,
based on the optimality of $U$,
it is shown that
$H(g(k),s)\le H(f(g(k)),s)+O(1)$
for all $k\in\N$ and $s\in\X$.
Hence,
using \eqref{HbnpflHnp} and Theorem~\ref{time} we have
$H(\rest{\beta}{g(k)})\le H(f(g(k)),p_{g(k)})+O(1)\le f(g(k))+O(1)$
for all $k$ with $f(g(k))\ge L_M$.
Since $\beta$ is weakly Chaitin $T$-random,
we have
$Tg(k)\le H(\rest{\beta}{g(k)})+O(1)\le f(g(k))+O(1)$
for all $k$ with $f(g(k))\ge L_M$.
However, this contradicts the fact that
the function $Tg(k)-f(g(k))$ of $k$ is unbounded,
and the proof is completed.
\end{proof}

\subsection{The Proof of Theorem~\ref{main2}}
\label{Proof_of_Theorem_main2}

The implication (i) $\Rightarrow$ (ii) of Theorem~\ref{main2}
can be proved based on Theorem~\ref{ZVTwCTr-strictTcb} as follows.

\begin{proof}[Proof of (i) $\Rightarrow$ (ii) of Theorem~\ref{main2}]
In the case of $T=1$, the implication (i) $\Rightarrow$ (ii) of Theorem~\ref{main2} results in the implication (i) $\Rightarrow$ (ii) of Theorem~\ref{one-way II}.
Thus, we assume that $T$ is a computable real with $0<T<1$ in what follows.
Let $V$ and $W$ be optimal prefix-free machines, and let $f$ is an order function.
Suppose that there exists $c\in\N$ such that $\Dom W$ is reducible to $\Pf(Z_V(T))$ in query size $f(n)+c$.
Then, by considering the following procedure, we first see that $n<H(n,\rest{Z_V(T)}{f(n)+c})+O(1)$ for all $n\in\N$.

Given $n$ and $\rest{Z_V(T)}{f(n)+c}$, one first calculates the finite set $\rest{\Dom W}{n}$.
This is possible since $\Dom W$ is reducible to $\Pf(Z_V(T))$ in query size $f(n)+c$ and $f(k)\le f(n)$ for all $k\le n$.
Then, by calculating the set $\{\,W(p)\mid p\in\rest{\Dom W}{n}\}$ and picking any one finite binary string $s$ which is not in this set,
one can obtain $s\in\X$ such that $n<H_W(s)$.

Thus, there exists a partial recursive function $\Psi\colon\N\times\X\to\X$ such that, for all $n\in\N$, $n<H_W(\Psi(n,\rest{Z_V(T)}{f(n)+c}))$.
It follows
from the optimality of $W$
that
\begin{equation}\label{n/T-f(n)-lower-bound}
  n<H(n,\rest{Z_V(T)}{f(n)+c})+O(1)
\end{equation}
for all $n\in\N$.

Now, let us assume contrarily that the function $n/T-f(n)$ of $n\in\N$ is unbounded.
Recall that $f$ is an order function and $T$ is computable.
Hence it is easy to show that there exists a total recursive function $g\colon\N\to\N$ such that
the function $f(g(k))$ of $k$ is increasing and the function $g(k)/T-f(g(k))$ of $k$ is also increasing.
For clarity, we define a total recursive function $m\colon \N\to\N$ by $m(k)=f(g(k))+c$.
Since $m$ is
injective,
it is then easy to see that there exists a partial recursive function $\Phi\colon\N\to\N$ such that $\Phi(m(k))=g(k)$ for all $k\in\N$.
Therefore,
based on the optimality of $U$,
it is shown that
$H(g(k),\rest{Z_V(T)}{m(k)})
\le H(\rest{Z_V(T)}{m(k)})+O(1)$
for all $k\in\N$.
It follows from \eqref{n/T-f(n)-lower-bound} that
$g(k)<H(\rest{Z_V(T)}{m(k)})+O(1)$
for all $k\in\N$.
On the other hand,
since $T$ is a computable real with $0<T<1$,
it follows from Theorem~\ref{ZVTwCTr-strictTcb}
that $H(\rest{Z_V(T)}{n})\le Tn+O(1)$ for all $n\in\N$.
Therefore we have $g(k)<Tf(g(k))+O(1)$ for all $k\in\N$.
However, this contradicts the fact that
the function $g(k)/T-f(g(k))$ of $k$ is unbounded,
and the proof is completed.
\end{proof}

On the other hand,
the implication (ii) $\Rightarrow$ (i) of Theorem~\ref{main2}
follows immediately from
Theorem~\ref{ZVThaltC} below and
Proposition~\ref{observations} (ii).

\begin{theorem}\label{ZVThaltC}
Suppose that $T$ is a computable real with $0<T\le 1$.
Let $V$ be an optimal prefix-free machine,
and let $F$ be a prefix-free machine.
Then $\Dom F$ is reducible to $\Pf(Z_V(T))$ in query size $\lceil n/T\rceil+O(1)$.
\end{theorem}

\begin{proof}
In the case where $\Dom F$ is a finite set,
the result is obvious.
Thus,
in what follows,
we assume that $\Dom F$ is an infinite set.

Let $p_0,p_1,p_2,p_3, \dotsc$ be
a particular recursive enumeration of $\Dom F$,
and let $G$ be a prefix-free machine such that
$\Dom G=\Dom F$ and $G(p_i)=i$ for all $i\in\N$. 
Recall here that we identify $\X$ with $\N$.
It is also easy to see that such a prefix-free machine $G$ exists.
Since $V$ is an optimal prefix-free machine,
from the definition of optimality of a prefix-free machine
there exists $d\in\N$ such that, for every $i\in\N$,
there exists $q\in\X$ for which $V(q)=i$ and $\abs{q}\le \abs{p_i}+Td$.
Thus,
$H_V(i)\le\abs{p_i}+Td$ for every $i\in\N$.
For each $s\in \X$,
we define $Z_V(T;s)$ as $\sum_{V(p)=s}2^{-\abs{p}/T}$.
Then, for each $i\in\N$,
\begin{equation}\label{imp}
  Z_V(T;i)\ge 2^{-H_V(i)/T}\ge 2^{-\abs{p_i}/T-d}.
\end{equation}
Then, by the following procedure,
we see that
$\Dom F$ is reducible to $\Pf(Z_V(T))$ in query size $\lceil n/T\rceil+d$.

Given $s\in\X$,
one first calculates $\rest{Z_V(T)}{\lceil n/T\rceil+d}$
by putting the queries $t$ to the oracle $\Pf(Z_V(T))$
for all $t\in\{0,1\}^{\lceil n/T\rceil+d}$, where $n=\abs{s}$.
Note here that
all the queries are of length $\lceil n/T\rceil+d$.
One then find $k_e\in\N$ such that
$\sum_{i=0}^{k_e} Z_V(T;i)>0.(\rest{Z_V(T)}{\lceil n/T\rceil+d})$.
This is possible because
$0.(\rest{Z_V(T)}{\lceil n/T\rceil+d})<Z_V(T)$,
$\lim_{k\to\infty}\sum_{i=0}^k Z_V(T;i)=Z_V(T)$,
and $T$ is a computable real.
It follows that
\begin{align*}
  \sum_{i=k_e+1}^{\infty} Z_V(T;i)
  &=Z_V(T)-\sum_{i=0}^{k_e} Z_V(T;i)
  <Z_V(T)-0.(\rest{Z_V(T)}{\lceil n/T\rceil+d})\\
  &<2^{-\lceil n/T\rceil -d}\le 2^{-n/T-d}.
\end{align*}
Therefore, by \eqref{imp},
\begin{equation*}
  \sum_{i=k_e+1}^{\infty} 2^{-\abs{p_i}/T}
  \le 2^{d}\sum_{i=k_e+1}^{\infty} Z_V(T;i)
  <2^{-n/T}.
\end{equation*}
It follows that, for every $i>k_e$,
$ 2^{-\abs{p_i}/T}<2^{-n/T}$
and therefore $n<\abs{p_i}$.
Hence,
$\rest{\Dom F}{n}
=\{\,p_i\mid i\le k_e\;\&\;\abs{p_i}\le n\,\}$.
Thus, one can calculate the finite set $\rest{\Dom F}{n}$.
Finally,
one accepts if $s\in\rest{\Dom F}{n}$ and rejects otherwise.
\end{proof}

\section{Concluding Remarks}
\label{conclusion}

Suppose that $T$ is a computable real with $0<T\le 1$.
Let $V$ be an optimal prefix-free machine.
It is worthwhile to clarify
the origin of the difference of the behavior of $Z_V(T)$ between $T=1$ and $T<1$ with respect to the notion of reducibility in query size $f$.
In the case of $T=1$,
the Ample Excess Lemma \cite{MY08} (i.e., Theorem~\ref{AEL}) plays a major role in establishing the unidirectionality of the computation from $\Omega_V$ to $\Dom V$.
However, in the case of $T<1$, this is not true because the weak Chaitin $T$-randomness of a real $\alpha$
does not necessarily imply that $\sum_{n=1}^{\infty} 2^{Tn-H(\reste{\alpha}{n})}<\infty$ \cite{RS05}.
On the other hand, in the case of $T<1$,
Lemma~\ref{RS-CHS} (i) plays a major role in establishing the bidirectionality of the computations between $Z_V(T)$ and $\Dom V$.
However, this does not hold for the case of $T=1$.

\section*{Acknowledgments}

This work was supported
by KAKENHI (20540134) and KAKENHI (23340020),
by CREST from Japan Science and Technology Agency,
and by the Ministry of Economy, Trade and Industry of Japan.



\appendix

\section{The proof of Theorem~\ref{time}}
\label{proof-time}

We here prove Theorem~\ref{time}.
For that purpose, we need Lemma~\ref{longer-running-time} below.
Let $V$ be an optimal prefix-free machine, and let $M$ be a deterministic Turing machine which computes $V$
throughout this Appendix~\ref{proof-time}.

\begin{lemma}\label{longer-running-time}
There exists $d\in\N$ such that,
for every $p\in\Dom V$,
there exists $q\in\Dom V$
for which $\abs{q}\le \abs{p}+d$ and
$T_M(q)>T_M(p)$.
\end{lemma}

\begin{proof}
Consider the prefix-free machine $F$ such that
(i) $\Dom F=\Dom V$ and
(ii) for every $p\in\Dom V$,
$F(p)=1^{2\abs{p}+T_M(p)+1}$.
It is easy to see that such a prefix-free machine $F$ exists.
Then,
since $V$ is an optimal prefix-free machine,
from the definition of an optimal prefix-free machine
there exists $d_1\in\N$ with the following property;
if $p\in\Dom F$, then there is $q$ for which
$V(q)=F(p)$ and $\abs{q}\le\abs{p}+d_1$.

Thus,
for each $p\in\Dom V$ with $\abs{p}\ge d_1$,
there is $q$ for which
$V(q)=F(p)$ and $\abs{q}\le\abs{p}+d_1$.
It follows that
\begin{equation}\label{output-input}
  \abs{V(q)}=2\abs{p}+T_M(p)+1>\abs{p}+d_1+T_M(p)\ge\abs{q}+T_M(p).
\end{equation}
Note that
exactly $\abs{q}$ cells on the tapes of $M$ have the symbols $0$ or $1$
in the initial configuration of $M$ with the input $q$,
while at least $\abs{V(q)}$ cells on the tape of $M$,
on which the output is put, have the symbols $0$ or $1$
in the resulting final configuration of $M$.
Since $M$ can write at most one $0$ or $1$ on the tape,
on which an output is put, every one step of its computation,
the running time $T_M(q)$ of $M$ on the input $q$
is bounded to the below by the difference $\abs{V(q)}-\abs{q}$.
Thus, by \eqref{output-input}, we have $T_M(q)>T_M(p)$.

On the other hand,
since $\Dom V$ is not a recursive set,
the function $T^M_n$ of $n\ge L_M$ is not bounded to the above.
Therefore, there exists $r_0\in\Dom V$ such that,
for every $p\in\Dom F$ with $\abs{p}<d_1$, $T_M(r_0)>T_M(p)$.
By setting $d_2=\abs{r_0}$
we then see that, for every $p\in\Dom F$ with $\abs{p}<d_1$,
$\abs{r_0}\le\abs{p}+d_2$.

Thus,
by setting $d=\max\{d_1,d_2\}$
we see that, for every $p\in\Dom V$,
there is $q\in\Dom V$ for which
$\abs{q}\le\abs{p}+d$ and $T_M(q)>T_M(p)$.
This completes the proof.
\end{proof}

Then the proof of Theorem~\ref{time} is given as follows.

\begin{proof}[Proof of Theorem~\ref{time}]
By considering the following procedure,
we first show that
$n\le H(n,p)+O(1)$ for all $(n,p)$ with $n\ge L_M$ and $p\in I_M^n$.

Given
$(n,p)$ with $n\ge L_M$ and $p\in I_M^n$,
one first calculates the finite set $\rest{\Dom V}{n}$
by simulating the computation of $M$ with the input $q$
until at most $T_M(p)$ steps,
for each $q\in\X$ with $\abs{q}\le n$.
Then,
by calculating the set $\{\,V(q)\mid q\in\rest{\Dom V}{n}\}$
and picking any one finite binary string $s$ which is not in this set,
one can obtain $s\in\X$ such that $n<H_V(s)$.

Hence, there exists a partial recursive function
$\Psi\colon \N^+\times\X\to \X$
such that,
for all $(n,p)$ with $n\ge L_M$ and $p\in I_M^n$,
$n<H_V(\Psi(n,p))$.
It follows
from the optimality of $V$ and $U$
that
$n<H(n,p)+O(1)$
for all $(n,p)$ with $n\ge L_M$ and $p\in I_M^n$.

We next show that $H(n,p)\le H(p)+O(1)$
for all $(n,p)$ with $n\ge L_M$ and $p\in I_M^n$.
From Lemma~\ref{longer-running-time} we first note that
there exists $d\in\N$ such that, for every $p\in\Dom V$,
there exists $q\in\Dom V$
for which $\abs{q}\le \abs{p}+d$ and $T_M(q)>T_M(p)$.
Then,
for each $(n,p)$ with $n\ge L_M$ and $p\in I_M^n$,
$\abs{p}\le n$ due to the definition of $I_M^n$, 
and also there exists $q\in\Dom V$
for which $\abs{q}\le \abs{p}+d$ and $T_M(q)>T_M(p)$.
Note here that
$T_M(q)>T_M^n$ due to the the definition of $I_M^n$ again,
and therefore $\abs{q}>n$ due to the definition of $T_M^n$.
Thus $\abs{p}\le n<\abs{p}+d$ and $d\ge 1$.
Hence,
given $p$ such that $n\ge L_M$ and $p\in I_M^n$,
one needs only $\lceil \log_2 d\rceil$ bits more to determine $n$,
since there are still only $d$ possibilities of $n$,
given such a string $p$.

Thus, there exists a partial recursive function
$\Phi\colon \X\times\X\to\N^+\times \X$ such that,
for every $(n,p)$ with $n\ge L_M$ and $p\in I_M^n$,
there exists $s\in\X$ with the properties that
$\abs{s}=\lceil \log_2 d\rceil$ and $\Phi(p,s)=(n,p)$.
It follows
that
$H(n,p)\le
H(p)+
\max\{H(s)\mid s\in\X\ \&\ \abs{s}=\lceil \log_2 d\rceil\}+O(1)$
for all $(n,p)$ with $n\ge L_M$ and $p\in I_M^n$.

Finally, we show that
$H(p)\le n+O(1)$ for all $(n,p)$ with $n\ge L_M$ and $p\in I_M^n$.
Let us consider the prefix-free machine $F$ such that
(i) $\Dom F=\Dom V$ and
(ii) for every $p\in\Dom V$, $F(p)=p$.
Obviously, such a prefix-free machine $F$ exists.
Then
we see that,
for every $p\in\Dom V$, $H(p)\le\abs{p}+O(1)$.
For each $(n,p)$ with $n\ge L_M$ and $p\in I_M^n$,
it follows from the definition of $I_M^n$
that
$p\in\Dom V$ and $\abs{p}\le n$,
and therefore we have $H(p)\le\abs{p}+O(1)\le n+O(1)$.
This completes the proof.
\end{proof}

\section{The proof of Lemma~\ref{RS-CHS}}
\label{proof-RS-CHS}

\begin{proof}[Proof of Lemma~\ref{RS-CHS}]
Let $T$ be a real with $T>0$, and let $V$ be an optimal prefix-free machine.

(i)
Chaitin \cite{C75} showed that
\begin{equation}\label{most_non-trivial}
  H(s,t)=H(s)+H(t/s)+O(1)
\end{equation}
for all $s,t\in\X$.
This
is Theorem~3.9.~(a) in \cite{C75}.
For the definition of $H(s/t)$, see Section~2 of Chaitin~\cite{C75}.
We here only use the property that,
for every $s\in\X$ and every $n\in\N$,
there exists $t\in\{0,1\}^n$ such that
\begin{equation}\label{simple_counting}
  H(t/s)\ge n.
\end{equation}
This is easily shown from the definition of $H(t/s)$
by counting the number of binary strings of length less than $n$.

On the other hand, it is easy to show that
\begin{equation}\label{Hstabst=Hst}
  H(st,\abs{t})=H(s,t)+O(1)
\end{equation}
for all $s,t\in\X$.
Since $H(st)+H(\abs{t})\ge H(st,\abs{t})-O(1)$ for all $s,t\in\X$,
it follows from
\eqref{Hstabst=Hst}, \eqref{most_non-trivial}, and \eqref{simple_counting}
that there exists $d\in\N$ such that,
for every $s\in\X$ and every $n\in\N$,
there exists $t\in\{0,1\}^n$ for which
$H(st)\ge H(s)+n-H(n)-d$.
Using
the optimality of $U$ and $V$,
we then see that there exists $d'\in\N$ such that,
for every $s\in\X$ and every $n\in\N$,
there exists $t\in\{0,1\}^n$ for which
\begin{equation}\label{HVst=HVs+n-Hn-d'}
  H_V(st)\ge H_V(s)+n-H(n)-d'.
\end{equation}

Now, suppose that $T<1$.
It follows
from the optimality of $U$
that
$H(n)\le 2\log_2 n+O(1)$ for all $n\in\N^+$.
Therefore there exists $c\in\N^+$ such that $(1-T)c-H(c)-d'\ge 0$.
Hence, by \eqref{HVst=HVs+n-Hn-d'} we see that
there exists $c\in\N^+$ such that, for every $s\in\X$,
there exists $t\in\{0,1\}^c$ for which $H_V(st)\ge H_V(s)+Tc$.

(ii)
Since $V$ is optimal,
it is easy to show that
there exists $d\in\N$ such that,
for every $s\in\X$ and every $n\in\N$,
\begin{equation}\label{HVs01nleHs+Hn+d}
\begin{split}
  H_V(s0^n)&\le H_V(s)+H(n)+d,\\
  H_V(s1^n)&\le H_V(s)+H(n)+d.
\end{split}
\end{equation}
Since $T>0$,
it follows
from the optimality of $U$
that
there exists $c\in\N^+$ such that $H(c)+d\le Tc-1$.
Hence, by \eqref{HVs01nleHs+Hn+d} we see that
there exists $c\in\N^+$ such that, for every $s\in\X$,
$H_V(s0^c)\le H_V(s)+Tc-1$ and $H_V(s1^c)\le H_V(s)+Tc-1$.
\end{proof}

\end{document}